\documentclass[11pt,reqno]{amsart}

\usepackage{amsmath,amsthm,amssymb,tocvsec2}
\usepackage[tmargin=1.75in,bmargin=1.55in,rmargin=1.4in,lmargin=1.4in]{geometry}
\usepackage[breaklinks=true]{hyperref}

\theoremstyle{plain}

\newtheorem{theorem}{Theorem}[section]
\newtheorem{corollary}[theorem]{Corollary}
\newtheorem{proposition}[theorem]{Proposition}
\newtheorem{lemma}[theorem]{Lemma}

\theoremstyle{definition}

\newtheorem{definition}[theorem]{Definition}

\newtheorem{example}[theorem]{Example}
\newtheorem{remark}[theorem]{Remark}

\numberwithin{equation}{section}
\numberwithin{table}{section}

\newcommand{\adj}{\mathop{\mathrm{adj}}}
\newcommand{\bc}{\mathbf{c}}

\newcommand{\bm}{\mathbf{m}}
\newcommand{\bp}{\mathcal{P}}
\newcommand{\bs}{\mathbf{s}}
\newcommand{\bu}{\mathbf{u}}
\newcommand{\bv}{\mathbf{v}}
\newcommand{\bw}{\mathbf{w}}
\newcommand{\bx}{\mathbf{x}}
\newcommand{\by}{\mathbf{y}}
\newcommand{\bone}{\mathbf{1}}
\newcommand{\Id}{\mathrm{Id}}
\newcommand{\meas}{\mathop{\mathrm{Meas}}\nolimits}
\newcommand{\moment}{\mathcal{M}}
\newcommand{\momentr}{\moment^{\rho}}
\newcommand{\momentrp}{\moment^{\rho,+}}
\newcommand{\sgn}{\mathop{\mathrm{sgn}}}
\newcommand{\std}{\,\mathrm{d}}

\newcommand{\C}{\mathbb{C}}
\newcommand{\F}{\mathbb{F}}
\newcommand{\R}{\mathbb{R}}
\newcommand{\Z}{\mathbb{Z}}
\newcommand{\nnZ}{\Z_+}

\newcommand{\cH}{\mathcal{H}}
\newcommand{\cL}{\mathcal{L}}
\newcommand{\cX}{\mathcal{X}}
\newcommand{\cY}{\mathcal{Y}}
\newcommand{\cZ}{\mathcal{Z}}

\newcommand{\HTN}{\cH^{++}}

\begin{document}
\title{Moment-sequence transforms}

\dedicatory{To Gadadhar Misra, master of operator theory}

\author{Alexander Belton}
\address[A.~Belton]{Department of Mathematics and Statistics, Lancaster
University, Lancaster, UK}
\email{\tt a.belton@lancaster.ac.uk}

\author{Dominique Guillot}
\address[D.~Guillot]{University of Delaware, Newark, DE, USA}
\email{\tt dguillot@udel.edu}

\author{Apoorva Khare}
\address[A.~Khare]{Department of Mathematics, Indian Institute of
Science, Bangalore, India and Analysis \& Probability Research Group,
Bangalore, India}
\email{\tt khare@iisc.ac.in}

\author{Mihai Putinar}
\address[M.~Putinar]{University of California at Santa Barbara, CA,
USA and Newcastle University, Newcastle upon Tyne, UK} 
\email{\tt mputinar@math.ucsb.edu, mihai.putinar@ncl.ac.uk}

\date{8th September 2021}

\keywords{Hankel matrix,
moment problem,
positive definite matrix,
totally non-negative matrix,
entrywise function,
absolutely monotonic function,
Laplace transform,
positive polynomial,
facewise absolutely monotonic function}

\subjclass[2010]{15B48 (primary);
30E05, 44A60, 26C05 (secondary)}

\begin{abstract}
We classify all functions which, when applied term by term, leave
invariant the sequences of moments of positive measures on the real line.
Rather unexpectedly, these functions are built of absolutely monotonic
components, or reflections of them, with possible discontinuities at
the endpoints. Even more surprising is the fact that functions preserving
moments of three point masses must preserve moments of all measures. Our
proofs exploit the semidefiniteness of the associated Hankel matrices and
the complete monotonicity of the Laplace transforms of the underlying
measures. As a byproduct, we characterize the entrywise
transforms which preserve totally non-negative Hankel matrices,
and those which preserve all totally non-negative
matrices. The latter class is surprisingly rigid: such maps must
be constant or linear. We also examine transforms in the
multivariable setting, which reveals a new class of piecewise
absolutely monotonic functions.
\end{abstract}
\maketitle

\settocdepth{section}
\tableofcontents

\section{Introduction}

The ubiquitous encoding of functions or measures into discrete entities,
such as sampling data, Fourier coefficients, Taylor coefficients,
moments, and Schur parameters, leads naturally to operating directly on
the latter `spectra' rather than the original. The present article
focuses on operations which leave invariant power moments of positive
multivariable measures. To put our essay in historical perspective, we
recall a few similar and inspiring instances.

The characterization of positivity preserving analytic operations on the
spectrum of a self-adjoint matrix is due to Loewner in his
groundbreaking article~\cite{Loewner34}. Motivated by the then-novel
theory of the Gelfand transform and the Wiener--Levy theorem, in the
1950s Helson, Kahane, Katznelson, and Rudin identified all real functions
which preserve Fourier transforms of integrable functions or measures on
abelian groups~\cite{HKKR,Kahane-Rudin,Rudin59}. Roughly speaking, these
Fourier-transform preservers have to be analytic, or even absolutely
monotonic. The absolute-monotonicity conclusion was not new, and
resonated with earlier work of Bochner~\cite{Bochner-pd} and
Schoenberg~\cite{Schoenberg42} on positive definite functions on
homogeneous spaces. Later on, this line of thought was continued by Horn
in his doctoral dissertation~\cite{horn}. These works all address the
question of characterizing real functions $F$ which have the property
that the matrix $(F(a_{ij}))$ is positive semidefinite whenever
$(a_{ij})$ is, possibly with some structure imposed on these matrices.
Schoenberg's and Horn's theorems deal with all matrices, infinite and
finite, respectively, while Rudin et al.~deal with Toeplitz-type matrices
via results of Herglotz and Carath\'eodory.

In this article, we focus on functions which preserve moment sequences of
positive measures on Euclidean space, or, equivalently, in the
one-variable case, functions which leave invariant positive semidefinite
Hankel kernels. As we show, these moment preservers are quite rigid, with
analyticity and absolute monotonicity again being present in a variety of
combinations, especially when dealing with multivariable moments. We
state in detail in Section~\ref{Sprelim} our results for one-variable
functions and domains and for moment sequences of measures on them, but
first we present in Section~\ref{Sconclude} tabulated lists of our
results in one and several variables.

The first significant contribution below is the relaxation to a minimal
set of conditions, which are very accessible numerically, that
characterize the positive definite Hankel kernel transformers in one
variable. Specifically, Schoenberg proved that a continuous map
$F : (-1,1) \to \R$ preserves positive semidefiniteness when applied
to matrices of all dimensions, if and only if $F$ is analytic and has
positive Taylor coefficients~\cite{Schoenberg42}. Later on, Rudin was
able to remove the continuity assumption~\cite{Rudin59}. In our first
major result, we prove that a map $F : (-1,1) \to \R$ preserves positive
semidefiniteness of all matrices if and only if it preserves this on
Hankel matrices. Even more surprisingly, a refined analysis reveals
that preserving positivity on Hankel matrices of rank at most~$3$ already
implies the same conclusion.

Our result can equivalently be stated in terms of preservers of moment
sequences of positive measures. Thus we also characterize such preservers
under various constraints on the support of the measures. Furthermore, we
examine the analogous problem in higher dimensions. In this situation,
extra work is required to compensate for the failure of Hamburger's
theorem in higher-dimensional Euclidean spaces.

Our techniques extend naturally to totally non-negative matrices,
in parallel to their natural connection to the Stieltjes moment
problem. We prove that the entrywise transformations which preserve
total non-negativity for all rectangular matrices, or all symmetric
matrices, are either constant or linear. Furthermore, we show that the
entrywise preservers of totally non-negative Hankel matrices must be
absolutely monotonic on the positive semi-axis. The class of totally
non-negative matrices was isolated by M.~Krein almost a century ago; he
and his collaborators proved its significance for the study of
oscillatory properties of small harmonic vibrations in linear elastic
media~\cite{GK, GK1}. Meanwhile this chapter of matrix analysis has
reached maturity and it continues to be explored and enriched on
intrinsic, purely algebraic grounds~\cite{FJ,FJS}.

We conclude by classifying transformers of tuples of moment sequences,
from which a new concept emerges, that of a piecewise absolutely
monotonic function of several variables. In particular, our results
extend original theorems by Schoenberg and Rudin. For more on the
wider framework within which this article sits, we refer the reader to
the survey~\cite{BGKP-survey}.

Besides the classical works cited above delineating this area of
research, we rely in the sequel on Bernstein's theory of absolutely
monotone functions~\cite{Bernstein,Widder}, a related pioneering article
by Lorch and Newman~\cite{Lorch-Newman} and Carlson's interpolation
theorem for entire functions~\cite{Carlson}.

The study of positive definite functionals defined on $*$-semigroups,
with or without unit, led Stochel to a series of groundbreaking
discoveries, complementing the celebrated Naimark and Sz.~Nagy
dilation theorems and, in particular, putting multivariate moment
problems in a wider, more flexible framework
\cite{Stochel85,Stochel91,Stochel92}. A byproduct of his studies is a
classification of positive definite functionals on the multiplicative
semigroup $( -1, 1 )$ \cite{Stochel91}, culminating with a similar
conclusion to our main one-dimensional result: these positive
functionals are absolutely monotonic on $( 0, 1 )$ with possibly
discontinuous derivatives, of any order, at the origin.

As a final remark, we note that entrywise transforms of moment sequences
were previously studied in a particular setting motivated by infinite
divisibility in probability theory~\cite{Horn-toeplitz, Stochel}. The
study of entrywise operations which leave invariant the cone of all
positive matrices has also recently received renewed attention in the
statistics literature, in connection to the analysis of big data. In that
setting, functions are applied entrywise to correlation matrices to
improve properties such as their conditioning, or to induce a Markov
random-field structure. The interested reader is referred
to~\cite{BGKP-fixeddim, GKR-lowrank, Guillot_Rajaratnam2012b} and the
references therein for more details.

A companion to the present article \cite{BGKP-TN} was recently
completed, which extends the work here with definitive classifications
of preservers of totally positive and totally non-negative kernels,
and together with kernels having additional structure, such as those
of Hankel \cite{Widder34} or Toeplitz \cite{Schoenberg51} type, or
generating series, such as P\'olya frequency functions and sequences.

\subsection{Summary of main results}\label{Sconclude}

Tables~\ref{T1var} and~\ref{Tmultivar} below summarize the results
proved in this article. The notation used below is explained in
the main body of the article; see also the List of Symbols following this
subsection.

In the one-variable setting, we have identified the positivity
preservers acting on (i)~all matrices, and (ii)~all Hankel matrices, in
the course of classifying such functions acting on (iii)~moment
sequences, i.e., all Hankel matrices arising from moment sequences of
measures supported on~$[-1,1]$.  Characterizations for all three classes
of matrices are obtained with the additional constraint that
the entries of the matrices lie in $(0,\rho)$, $(-\rho,\rho)$, and
$[0,\rho)$, where $\rho \in (0,\infty]$.

{\renewcommand{\arraystretch}{1.3}
\begin{table}[h]
\begin{tabular}{|c|c|c|c|}
\hline
& \multicolumn{3}{c|}{$F[-]$ preserves positivity on:}\\ \hline
Domain $I$, & $\cup_{N \geq 1} \bp_N(I)$ & $\cH^+(I)$ & $\mu \in
\moment([0,1])$ or $\moment([-1,1])$,\\
$\rho \in (0,\infty]$ & & & $s_0(\mu) \in I \cap [0,\infty)$\\
\hline \hline
$(0,\rho)$ & Theorems~\ref{Thorn}, \ref{Treformulation}
& Theorems~\ref{Thorn-hankel}, \ref{Treformulation} &
Theorems~\ref{Thorn-hankel}, \ref{Treformulation}\\ \hline
$[0,\rho)$ & Proposition~\ref{P1sided} & Proposition~\ref{P1sided} &
Theorem~\ref{T1sided-general}\\ \hline
$(-\rho,\rho)$ & Theorem~\ref{TSchoenberg} &
Theorem~\ref{T2sided-general} &
Theorem~\ref{T2sided-general}\\ \hline
\end{tabular}\vspace*{2mm}
\caption{The one-variable case.}\label{T1var}
\end{table}}

We then extend each of the results in Table~\ref{T1var} to apply to
functions acting on tuples of positive matrices or moment sequences:
see Table~\ref{Tmultivar}.

{\renewcommand{\arraystretch}{1.3}
\begin{table}[h]
\begin{tabular}{|c|c|c|c|}
\hline
& \multicolumn{3}{c|}{$F[-]$ preserves positivity on $m$-tuples of
elements in:}\\ \hline
Domain $I$, & $\cup_{N \geq 1} \bp_N(I)$ & $\cH^+(I)$ & $\mu \in
\moment([0,1])$ or $\moment([-1,1])$,\\
$\rho \in (0,\infty]$ & & & $s_0(\mu) \in I \cap [0,\infty)$\\
\hline \hline
$(0,\rho)$ & Theorem~\ref{Thorn-hankel2} & Theorem~\ref{Thorn-hankel2}  &
Theorem~\ref{Thorn-hankel2}\\ \hline
$[0,\rho)$ & Proposition~\ref{P1sided-multi} &
Proposition~\ref{P1sided-multi} & Theorem~\ref{T1sided-multi}\\ \hline
$(-\rho,\rho)$ & Theorem~\ref{T2sided-multi} &
Theorem~\ref{T2sided-multi} & Theorem~\ref{T2sided-multi}\\
& (see \cite{fitzgerald} for $\rho = \infty$)  & &\\ \hline
\end{tabular}\vspace*{2mm}
\caption{The multivariable case.}\label{Tmultivar}
\end{table}}

In the one-variable setting, we do more than is recorded in
Table~\ref{T1var}, since our results cover various classes of totally
non-negative matrices (Section~\ref{TN}), as well as the
closed-interval settings of $[0,\rho]$ and $[-\rho,\rho]$ for
$\rho < \infty$ (Section~\ref{Sdomain}). The multivariable case may
contain products of open and closed intervals, but it would be rather
cumbersome, and somewhat artificial, to consider them all. We do not
pursue this direction in the present work.

In all of the above contexts, with the exception of functions
on~$[0,\rho)^m$ (i.e., the $(2,3)$ entry in both tables), the
characterizations are uniform: all such positivity preservers are
necessarily analytic on the domain and absolutely monotonic on the
closed positive orthant. The converse result holds trivially by the
Schur product theorem. The one exceptional case reveals a richer
family of `facewise absolutely monotonic maps'; see
Section~\ref{SfaceAM}.

We have also improved on all of the above results, by significantly
relaxing the hypotheses required to obtain absolute monotonicity.

Finally, and for completeness, we remark that
Theorem~\ref{Tthreshold2} from our previous work \cite{BGKP-fixeddim},
which is widely used herein, admits a generalization to all, possibly
non-consecutive, integer powers, and again the bounds have closed
form. This result is obtained through a careful analysis and novel
results about Schur polynomials; we refer the reader to the recent
paper by Khare and Tao \cite{KT} for more details.

\subsection{List of symbols}

For the convenience of the reader, we list some of the symbols used in
this paper.

\begin{itemize}
\item Given a subset $I \subset \R$, $\bp_N^k( I )$ is the set of
positive semidefinite $N \times N$ matrices with entries in $I$ and of
rank at most $k$. We let $\bp_N( I ) := \bp_N^N( I )$
and $\bp_N := \bp_N( \R )$.

\item $\cH^+(I)$ denotes the set of positive semidefinite Hankel
matrices of arbitrary dimension with entries in $I$.

\item $\HTN_n$ denotes the set of $n \times n$ totally
non-negative Hankel matrices, and $\HTN$ denotes the set of all
totally non-negative Hankel matrices.

\item $H^{(1)}$ denotes the truncation of a possibly semi-infinite matrix
$H$ obtained by excising the first column.

\item $F[ H ]$ is the result of applying $F$ to each entry of the
matrix~$H$.

\item For $K \subset \R$, we denote by $\meas^+(K)$ the set of
admissible measures, i.e., non-negative measures $\mu$ supported
on~$K$ and admitting moments of all orders.

\item The $k$th moment of a measure $\mu$ is denoted by~$s_k(\mu)$;
the corresponding moment sequence is
$\bs(\mu) := (s_k(\mu))_{k \geq 0}$. The associated Hankel moment
matrix $H_\mu$ has $(i,j)$ entry $s_{i+j}(\mu)$. In particular, the
moment sequence of~$\mu$ is the leading row and column of~$H_\mu$.

\item Given $K \subset \R$, $\moment(K)$ denotes the set of
moment sequences associated to elements of $\meas^+(K)$.
For any $k \geq 0$, $\moment_k(K)$ denotes the corresponding set of
truncated moment sequences:
$\moment_k(K) := \{ (s_0(\mu), \ldots, s_k(\mu)) : \mu \in \meas^+(K) \}$.

\item Given $K \subset \R$ and a scalar $\rho$ with
$0 < \rho \leq \infty$, $\momentr(K)$ denotes the subset of
$\moment(K)$ with moments $s_j \in (-\rho, \rho)$ for all $j \geq 0$,
and, for any  $k \geq 0$, we let $\momentr_k(K)$ denote the subset of
$\moment_k(K)$ with $s_j \in (-\rho, \rho)$ for $j = 0, \ldots, k$.

\item Given $\rho$ with $0 < \rho \leq \infty$, an
integer $k \geq 0$, and $x \in [-1,1)$, we let
$\momentrp(\{ 1, x \})$ and $\momentrp_k(\{ 1, x \})$ denote the
subsets of $\moment(\{ 1, x \})$ and $\moment_k(\{ 1, x \})$,
respectively, with total mass $s_0 < \rho$ and such that $1$ and $x$ both
have positive mass.

\item Given an integer~$m\geq1$, a function $F : \R^m \to \R$ acts on
tuples of moment sequences of admissible measures in
$\moment(K_1) \times \cdots \times \moment(K_m)$ as follows:
\begin{equation}
F[\bs(\mu_1), \dots, \bs(\mu_m)] :=
(F(s_k(\mu_1), \dots, s_k(\mu_m)))_{k \geq 0}.
\end{equation}

\item Given $h>0$ and an integer $n \geq 0$, $\Delta^n_h F$ denotes the
$n$th forward difference of the function~$F$ with step size~$h$.

\item $\bone_{m \times n}$ denotes the $m \times n$ matrix
with all entries equal to $1$.

\item $\C^+ := \{ z \in \C: \Re z > 0 \}$ denotes the right open
half-plane.
\end{itemize}

\subsection{Organization}

The plan of the article is as follows. Section~\ref{Sprelim} recalls
notation and reviews previous work, while Section~\ref{S1dmain} lists
our main results for classical positive Hankel matrices transformers,
which, in particular, go beyond previous classical results.
Sections~\ref{Smoments01}, \ref{Smomentsm11}, \ref{Smomentsm10},
and~\ref{Sdomain} are devoted to proofs, arranged by the domains of the
entries of the relevant Hankel matrices. For these proofs, we work with
measures with restricted total mass, which is reflected in the domains of
the test sets of matrices, and helps unify previously known results.
Thus, we end up showing stronger results than in Section~\ref{Sprelim};
these results were tabulated in a concise form in Section~\ref{Sconclude}
above. An additional strengthening involves severely reducing the
supports of the test measures, which translates to rank constraints on
the test sets of Hankel matrices and hence stronger results. This
technical point is not mentioned in the above tables, but is detailed in
the aforementioned
Sections~\ref{Smoments01},~\ref{Smomentsm11},~\ref{Smomentsm10},
and~\ref{Sdomain} devoted to proofs.

Section~\ref{TN} contains the classifications of preservers of
total non-negativity for several different sets of matrices, in the
dimension-free setting.
Section~\ref{Smulti} deals with multivariable transformers of Hankel
kernels. Section~\ref{SLaplace} makes the natural link with Laplace
transforms and interpolation of entire functions.
The appendix is devoted to algebraic properties of adjugate matrices.

\subsection{Acknowledgements}

The authors extend their thanks to the International Centre for
Mathematical Sciences, Edinburgh, where the major part of this work
was carried out.
D.G.~is partially supported by a University of
Delaware Research Foundation grant, by a Simons Foundation
collaboration grant for mathematicians, and by a University of
Delaware Research Foundation Strategic Initiative grant.
A.K.~is partially supported by
Ramanujan Fellowship SB/S2/RJN-121/2017,
MATRICS grant MTR/2017/000295, and
SwarnaJayanti Fellowship grants SB/SJF/2019-20/14 and DST/SJF/MS/2019/3
from SERB and DST (Govt.~of India),
by grant F.510/25/CAS-II/2018(SAP-I) from UGC (Govt.~of India),
by a Tata Trusts gravel grant,
and by a Young Investigator Award from the Infosys Foundation.
We are grateful to the referees for valuable comments and enriching
bibliographical indications.

\section{Preliminaries}\label{Sprelim}

We collect in this section the basic concepts and notation necessary
for accessing the rest of the article. Bibliographical indications
will rely on classical texts. We are fortunate to be able to refer to
a few very recent outstanding monographs, including
\cite{Schmuedgen,Simon}.

\subsection{Matrices of moments}

Our raw material consists of structured matrices of moments and functions
acting on them. In this subsection, we concentrate on the first.
Henceforth $N$ is a positive integer.

\begin{definition}
Given a subset $I \subset \R$, denote by $\bp_N(I)$ the set of
positive semidefinite $N \times N$ matrices with entries in~$I$, and
let $\bp_N := \bp_N( \R )$.
\end{definition}

The set $\bp_N$ is a convex cone, closed in the Euclidean topology of
$\R^{N \times N}$. Schur's product theorem asserts $A \circ B \in \bp_N$
whenever $A, B \in \bp_N$; here $A \circ B = (a_{ij} b_{ij})$ denotes the
entrywise product of two equidimensional matrices $A = (a_{ij})$ and $B =
(b_{ij})$. For a proof it is sufficient to decompose $B$ into a sum of
rank-one positive matrices and follow the definition of matrix
positivity.

Recall that a matrix is said to be \emph{totally non-negative} if all its
minors are non-negative. Totally non-negative matrices occur in a variety
of areas; see~\cite{FJ} and the references therein. For instance, a
well-known observation due to Schoenberg asserts that given vectors
$x_1,x_2, \ldots, x_N$, in an inner-product space, the corresponding
matrix $( \exp(- \| x_j - x_k \|^2)_{j,k=1}^N$ is totally non-negative.

\begin{definition}
For an integer $n \geq 1$, let $\HTN_n$ denote the set of $n \times n$
totally non-negative Hankel matrices, and let $\HTN := \bigcup_{n \geq 1}
\HTN_n$ denote the set of totally non-negative Hankel matrices of
arbitrary size.
\end{definition}

The moment problem, in the widely accepted meaning of the term, is
arguably the quintessential inverse problem. It has a long history and
continues to lead to unexpected impacts in pure and applied
mathematics; see, for
instance,~\cite{Akhiezer,Lasserre,Schmuedgen,STmoment}. Moments of
positive measures are in general observables, with a physical or
probabilistic interpretation. These observed real numbers are not
free, but are subject to an array of semi-algebraic constraints, which
are generally hard to deal with directly. A convenient and numerically
friendly approach is to organize the moments into matrices with
red[undant entries, the simplest case being associated to measures
supported on subsets of the real line. We will start with this generic
situation.

Let $\mu$ be a non-negative measure on $\R$, rapidly decreasing at
infinity, that admits moments of all orders; let its moment data and
associated Hankel matrix be denoted as follows:
\begin{equation}\label{Ehankel}
s_k(\mu) = s_k := \int_\R x^k \std \mu, \qquad
\bs(\mu) := (s_k(\mu))_{k \geq 0}, \qquad
H_\mu := \begin{pmatrix} s_0 & s_1 & s_2 & \cdots\\
s_1 & s_2 & s_3 & \cdots\\
s_2 & s_3 & s_4 & \cdots\\
\vdots & \vdots & \vdots & \ddots
\end{pmatrix}.
\end{equation}
All measures appearing in this paper are taken to be non-negative and
are assumed to possess moments of all orders. We will henceforth call
such measures \emph{admissible}.

Throughout this paper, we allow matrices to be semi-infinite in both
coordinates. We also identify without further comment the space of
real sequences $(s_0, s_1, \dots)$ and the corresponding Hankel
matrices, as done in~(\ref{Ehankel}).

To verify the positivity of the matrix $H_\mu$, it is sufficient to
observe that
\[
0 \leq \int |\sum_{j=0}^N c_j x^j|^2 \std \mu = \sum_{j,k=0}^N H_\mu(j,k)
c_j c_k.
\]

\begin{definition}
Given subsets $I$, $K \subset \R$, let~$\meas^+(K)$ denote the
admissible measures supported on~$K$, and let~$\cH^+(I)$ denote
the set of complex Hermitian positive semidefinite Hankel
matrices with entries in~$I$. We will henceforth use the adjective
`positive' to mean `complex Hermitian positive semidefinite' when
applied to matrices.
\end{definition}

The following theorem combines classical results of Hamburger,
Stieltjes, and Hausdorff. 

\begin{theorem}\label{Thamburger}
A sequence $\bs = (s_k)_{k=0}^\infty$ is a moment sequence for an
admissible measure on~$\R$ if and only if the Hankel matrix with first
column $\bs$ is positive. In other words, the map
$\Psi : \mu \mapsto (s_k(\mu))_{k=0}^\infty$ is a surjection from
$\meas^+(\R)$ onto $\cH^+(\R)$. Moreover,
\begin{enumerate}
\item restricted to $\meas^+([0,\infty))$, the map $\Psi$ is a
surjection onto the positive Hankel matrices with non-negative
entries, such that removing the first column still yields a positive
matrix;

\item restricted to $\meas^+([-1,1])$, the map $\Psi$ is a bijection
onto the positive Hankel matrices with uniformly bounded entries;

\item restricted to $\meas^+([0,1])$, the map $\Psi$ is a bijection
onto the positive Hankel matrices with uniformly bounded entries, such
that removing the first column still yields a positive matrix.
\end{enumerate}
\end{theorem}

\begin{proof}
The first assertion is classical; for example, see Akhiezer's book
\cite[Theorems 2.1.1, 2.6.4, and~2.6.5]{Akhiezer}. For the last two
statements, we simply remark that for an admissible measure $\mu$,
\[
s_{2 n}( \mu ) = \int_{[-1,1]} x^{2 n} \std\mu + %
\int_{\R \setminus [-1,1]} x^{2 n} \std\mu.
\]
The first integral remains uniformly bounded as a function of $n$, while
the second tends to infinity with $n$ whenever the measure $\mu$ has
positive mass on $\R \setminus [-1,1].$
\end{proof}

\begin{definition}\label{Dtruncate}
In view of the above correspondence, we denote by $\moment(K)$ the set
of moment sequences associated to measures in $\meas^+(K)$.
Equivalently, $\moment(K)$ is the collection of first columns of
Hankel matrices associated to admissible measures supported on~$K$.
We write $H^{(1)}$ to denote the truncation of a matrix $H$ in which
the first column is excised.
\end{definition}

For technical reasons which will become apparent from the proofs
below, we introduce an additional parameter via the following
definition.

\begin{definition}\label{Dmass}
Given $0 < \rho \leq \infty$ and $I \subset \R$, let $\momentr(I)$ denote
the set of moment sequences $(s_k(\mu))_{k=0}^\infty$ of admissible
measures $\mu$ supported on $I$, with all moments in
$(-\rho, \rho)$. Also, for any $n \geq 0$, let $\momentr_n(I)$ denote the
corresponding set of \emph{truncated moment sequences}
$(s_k(\mu))_{k=0}^n$.
\end{definition}

Note that $\momentr(I) = \moment(I)$ and $\momentr_n(I) = \moment_n(I)$
when $\rho = \infty$. Moreover, for a non-negative measure $\mu$
supported on $[-1,1]$, the mass $s_0(\mu)$ dominates $|s_k(\mu)|$ for all
$k \geq 0$. Studying moment sequences of admissible measures having mass
$s_0 < \rho$ is therefore equivalent to working with Hankel matrices with
entries in a bounded interval $(-\rho, \rho)$. This will be our approach
in the remainder of the paper.

A simple characterization of rank-one Hankel matrices is stated below.

\begin{lemma}\label{Lrank1Hankel}
A rank-one $N \times N$ matrix $\bu \bu^T$, with entries in any field,
is Hankel if and only if either the successive entries of~$\bu$ are in
a geometric progression, or all entries but the last are~$0$. More
precisely, the matrix $\bu \bu^T$ is Hankel if and only if
\begin{equation}\label{Erank1Hankel}
u_j = \begin{cases}
u_1 (u_2/u_1)^{j-1} \qquad & \text{if } u_1 \neq 0,\\
0 & \text{if } u_1 = 0 \text{ and } 1 \leq j < N.
\end{cases}
\end{equation}
\end{lemma}

\begin{proof}
This is immediate for $N \geq 2$. For $N>2$, each principal $3 \times 3$
block submatrix of $\bu \bu^T$ with successive rows and columns is of the
form
\[
\begin{pmatrix} u_{j-1}^2 & u_{j-1} u_j & u_{j-1} u_{j+1}\\
u_j u_{j-1} & u_j^2 & u_j u_{j+1}\\
u_{j+1} u_{j-1} & u_{j+1} u_j & u_{j+1}^2 \end{pmatrix},
\]
whence $u_{j-1} u_{j+1} = u_j^2$ for all $j \geq 2$.
Identity~(\ref{Erank1Hankel}) follows immediately.
\end{proof}

We invite the reader to find all positive measures on the real line
which produce a rank-one Hankel matrix. In general, one can read off
from a positive Hankel matrix whether the representing measure is
unique, and estimate the shape of the support of the representing
measure(s) (of utmost importance in polynomial optimization), and
enter into the Lebesgue decomposition of the representing
measure(s). We refer to~\cite{Akhiezer,Lasserre,Schmuedgen} for
aspects of such refined analysis pertaining to the moment problem and
its current applications.

In Section~\ref{Smulti}, we will treat multivariable moment
problems. In that context, Hankel matrices are replaced by kernels
with a Hankel-type property. The semigroup approach proves to be
superior in the multivariablee setting; see~\cite{BCR} for more
details.

To conclude, we note that the study of Hankel matrices forms an
important chapter of modern analysis, with ramifications for
approximation theory, probability theory and control
theory~\cite{Peller}.

\subsection{Absolutely monotonic functions}

We turn now to operators on moments by identifying two relevant classes
of functions.

Central to our study is the class of \emph{absolutely monotonic
entire functions}. These are entire functions with non-negative Taylor
coefficients at every point of $( 0, \infty )$. Equivalently, it is
sufficient for such a function to have non-negative Taylor coefficients
at zero. Their structure was unveiled in a fundamental memoir by
Bernstein~\cite{Bernstein}; see also Widder's book~\cite{Widder} or the
recent treatise~\cite{SSV}.

One can restrict the absolute monotonicity definition to a finite
interval, with the following outcome.

\begin{theorem}[{\cite[Chapter~IV, Theorem~3a]{Widder}}]\label{Twidder}
If $f$ is absolutely monotonic on~$[a,b)$, then it can be extended
analytically to the complex disc centered at~$a$ and of
radius~$b-a$.
\end{theorem}

Recall that a function is said to be \emph{completely monotonic} on an
interval~$(a,b)$ if the map $x \mapsto f(-x)$ is absolutely monotonic
on~$(-b,-a)$, i.e., if $(-1)^k f^{(k)}(x) \geq 0$ for all
$x \in (a,b)$.  Similarly, a function is completely monotonic on an
interval $I \subset \R$ if it is continuous on~$I$ and is completely
monotonic on the interior of $I$.

Complete monotonicity can also be defined using finite
differences. Let $\Delta^n_h f$ denote the $n$th forward difference
of~$f$ with step size~$h$:
\[
\Delta^n_h f(x) := \sum_{k=0}^n (-1)^{n-k} \binom{n}{k} f(x+kh).
\]
Then $f$ is completely monotonic on $(a,b)$ if and only if
$(-1)^n \Delta^n_h f( x ) \geq 0$ for all non-negative integers~$n$
and for all $x$, $h$ such that $a < x < x+h < \dots < x + nh < b$.
See \cite[Chapter~IV]{Widder} for more details on completely monotonic
functions. Such functions were also characterized in a celebrated
result of Bernstein.

\begin{theorem}[{Bernstein
\cite[Chapter~IV, Theorem~12a]{Widder}}]\label{Tbernstein}
A function $f: [0,\infty) \to \R$ is completely monotonic on
$0 \leq x < \infty$ if and only if 
\[
f(x) = \int_0^\infty e^{-xt}\std\mu(t)
\]
for some finite positive measure $\mu$.
\end{theorem}

Atomic measures are not excluded in Bernstein's theorem, hence series
of exponentials and Dirichlet series are an integral part of the
theory of absolutely or completely monotonic functions. One of the
major advantages of absolute monotonicity is the analytic extension of
the respective function to a complex domain. We will exploit this
quality further on in the present work.

\subsection{Matrix positivity transforms}

The main theme of our work is permanence properties of moment matrices
$A$ under \emph{entrywise} operations. From the very beginning we warn
the reader that our framework is in contrast to the classical
functional calculus $A \mapsto f(A)$ which is the subject of Loewner's
celebrated theorem: \emph{a real function $f$ preserves matrix
ordering (i.e., $A \leq B$ implies $f(A) \leq f(B)$) among
self-adjoint matrices if and only if $f$ extends analytically to the
upper-half plane and it has positive imaginary part there.} For ample
details and a dozen different proofs, see~\cite{Donoghue,Simon}.

Entrywise operations on matrices and kernels also have a long and
interesting history, see~\cite{BGKP-survey}. We will provide the
outlines of a few significant results.

Transformations which leave invariant Fourier transforms of various
classes of measures on groups or homogeneous spaces were studied by
many authors, including Schoenberg~\cite{Schoenberg42},
Bochner~\cite{Bochner-pd}, Helson, Kahane, Katznelson, and
Rudin~\cite{HKKR,Kahane-Rudin}. From the latter works, Rudin
extracted~\cite{Rudin59} an analysis of maps which preserve moment
sequences for admissible measures on the torus; equivalently, these are
functions which, when applied entrywise, leave invariant the cone of
positive semidefinite Toeplitz matrices. Rudin's result, originally
proved by Schoenberg~\cite{Schoenberg42} under a continuity assumption,
is as follows.

\begin{theorem}[Schoenberg, Rudin]\label{TSchoenberg}
Given a function $F : (-1,1) \to \R$, the following are equivalent.
\begin{enumerate}
\item Applied entrywise, $F$ preserves positivity on the space of
positive matrices with entries in $(-1,1)$ of all dimensions.

\item Applied entrywise, $F$ preserves positivity on the space of
positive Toeplitz matrices with entries in $(-1,1)$ of all dimensions.

\item The function $F$ is real analytic on~$(-1,1)$ and absolutely
monotonic on~$(0,1)$.
\end{enumerate}
\end{theorem}

The facts that $(3) \implies (1)$ and $(3) \implies (2)$ follow from the
Schur product theorem~\cite{Schur1911}. However, the converse results are
highly non-trivial.

In the present paper, we consider moments of measures on the line
rather than Fourier coefficients, so power moments rather than complex
exponential moments. Hence we study functions $F$ mapping moment
sequences entrywise into themselves, i.e., such that for every
admissible measure $\mu$, there exists an admissible measure
$\sigma = \sigma_\mu$ satisfying
\[
F(s_k(\mu)) = s_k(\sigma) \qquad \text{for all } k \geq 0.
\]
Equivalently, by Theorem \ref{Thamburger}, we study entrywise
endomorphisms of the cone of positive Hankel matrices with real
entries. The following notion of entrywise calculus is central to this
paper.

\begin{definition}
Given a domain $D \subset \R$ and a function $F : D \to \R$, the
function $F[-]$ acts on the set of matrices with entries in~$D$, by
applying $F$ entrywise:
\[
F[ A ] := (F(a_{ij})) \qquad \text{for the matrix } A = (a_{ij}).
\]
The function $F$ also acts entrywise on moment sequences with all
moments in $D$, so that $F[ \bs( \mu ) ]_k := F( s_k( \mu ) )$ for all
$k \geq 0$, and similarly for truncated moment sequences.
\end{definition}

An observation on positivity preservers made by Loewner and developed
by Horn~\cite{horn} provides the following necessary condition for a
function to preserve positivity on $\bp_N((0,\infty))$ when applied
entrywise.

\begin{theorem}[Horn]\label{Thorn}
If a continuous function $F : ( 0, \infty) \to \R$ is such that
$F[ - ]: \bp_N( ( 0, \infty ) ) \to \bp_N( \R )$, then
$F \in C^{N - 3}( ( 0, \infty ) )$ and $F^{( k )}( x ) \geq 0$ for all
$x > 0$ and all $0 \leq k \leq N - 3$. Moreover, if it is known that
$F \in C^{N - 1}( ( 0, \infty ) )$, then $F^{( k )}( x ) \geq 0$ for
all $x > 0$ and all $0 \leq k \leq N - 1$.
\end{theorem}

The main idea in the proof is to develop into Taylor series a
perturbation determinant
\[
\det [ F(a + t u_j u_k)]_{j,k=1}^N
\]
and isolate the first non-zero coefficient as a universal constant
times the product $F(a) F'(a) \cdots F^{(N-3)}(a)$. Our prior work in
fixed dimension has amply exploited the symmetry and combinatorial
flavor of similar determinants~\cite{BGKP-fixeddim}.

\section{Main results in 1D}\label{S1dmain}

We state in this brief section our main results, restricted to the
one-variable case. The proofs will be given in subsequent sections
with a gradual increase in technicality, which also applies the
statements of these results. A leading thread is the isolation of
minimal sets of matrices for the verification of preservers, without
altering the conclusions. We remind the reader that all functions in
this article act entry by entry on moment sequences and matrices.

The following theorem, the first in a series to be established below,
gives an idea of the type of positive Hankel-matrix preservers we
seek.

\begin{theorem}\label{T2sided}
A function $F : \R \to \R$ maps $\moment([-1,1])$ into itself when
applied entrywise, if and only if $F$ is the restriction
to~$\R$ of an absolutely monotonic entire function.
\end{theorem}

In particular, Theorem~\ref{T2sided} strengthens the Schoenberg--Rudin
Theorem~\ref{TSchoenberg}, by relaxing the assumptions in \cite{Rudin59,
Schoenberg42} to require positivity preservation only for Hankel
matrices. Theorem~\ref{T2sided} is proved in Section~\ref{Smomentsm11}
with three further strengthenings: we use test sets with at most three
points (corresponding to Hankel test matrices of rank at most three),
the measures are allowed to have a mass constraint, enabling us to
classify functions
$F : ( -\rho, \rho ) \to \R$, where $0 < \rho \leq \infty$, and we
show that allowing functions to map entrywise into the co-domain
$\moment(\R)$ does not enlarge the class of preservers.

Our next result is a one-sided variant of the above characterizations,
following Horn~\cite[Theorem 1.2]{horn}. Akin to Theorem~\ref{T2sided},
it arrives at the same conclusion under weaker assumptions than
in~\cite{horn}.

\begin{theorem}\label{T1sided}
A function $F : [0,\infty) \to \R$ maps $\moment([0,1])$ into itself
when applied entrywise, if and only if $F$ is absolutely monotonic on
$(0,\infty)$, so non-decreasing, and
$0 \leq F(0) \leq \lim_{\epsilon \to 0^+} F(\epsilon)$.
\end{theorem}

In Section~\ref{Smoments01}, we use results of Bernstein and
Lorch--Newman to prove Theorem~\ref{T1sided}, and then provide a
strengthening of it, Theorem~\ref{T1sided-general}, in the spirit
described above after Theorem~\ref{T2sided}. Here, we can replace
$\moment([0,1])$ by test measures supported on at most two points.

Next, we provide a classification of the preservers of
$\moment([0,\infty))$, Theorem~\ref{Ttn}, which gives a
Schoenberg-type characterization of functions preserving total
non-negativity. It is akin to Theorem \ref{T1sided}, and provides a
connection between moment sequences, totally non-negative Hankel
matrices, and their preservers; see Section~\ref{TN} for the proof.

\begin{theorem}\label{Ttn}
For a function $F : [0,\infty) \to \R$, the following are equivalent.
\begin{enumerate}
\item Applied entrywise, the function $F$ preserves positive
semidefiniteness on the set~$\HTN$ of all totally
non-negative Hankel matrices.

\item Applied entrywise, the function $F$ preserves the set
$\HTN$.

\item Applied entrywise, the function $F$ sends $\moment([0,\infty))$
to itself.

\item The function $F$ agrees on $(0,\infty)$ with an
absolutely monotonic entire function and
$0 \leq F(0) \leq \lim_{\epsilon \to 0^+} F(\epsilon)$.
\end{enumerate}
\end{theorem}

Our techniques lead to the following observation: the only
non-constant maps which preserve the set of all totally non-negative
matrices when applied entrywise are of the form $F( x ) = c x$,
where~$c > 0$. See Theorem~\ref{Ptn} for more details.

Returning to moment sequences, in the present paper we also study
preservers of $\moment( [ -1, 0 ] )$, and show that these are
classified as follows.

\begin{theorem}\label{Tminus}
The following are equivalent for a function $F : \R \to \R$.
\begin{enumerate}
\item Applied entrywise, $F$ maps $\moment([-1,0])$ into
$\moment((-\infty,0])$.

\item There exists an absolutely monotonic entire function
$\widetilde{F}$ such that
\[
F(x) = \begin{cases}
\widetilde{F}(x) & \text{if } x \in (0,\infty),\\
0 & \text{if } x=0,\\
-\widetilde{F}(-x) \qquad & \text{if } x \in (-\infty,0).
\end{cases}
\]
\end{enumerate}
\end{theorem}

It is striking to observe the possibility of a discontinuity at the
origin, in both of the previous theorems. For the proof of this
result, we refer the reader to Section~\ref{Smomentsm10}.

We also derive a similar description of the functions that
transform $\moment([-1,0])$ into $\moment([0,\infty))$; see
Theorem~\ref{Tminus-even}.  In this variant,  we observe that~$F$ may
also be discontinuous at~$0$.

The arguments used to show Theorem~\ref{TSchoenberg} and its one-sided
variant by Schoenberg, Rudin, and Horn do not carry over to our
setting involving positive Hankel matrices. This is due to the fact
that the hypotheses in Theorems~\ref{T2sided} and~\ref{T1sided} are
significantly weaker.

We show below how to further relax quite substantially the assumptions in
Theorem~\ref{T2sided} (Section~\ref{Smomentsm11}), Theorem~\ref{T1sided}
(Section~\ref{Smoments01}), and Theorem~\ref{Tminus}
(Section~\ref{Smomentsm10}). In doing so, our goal is to understand the
minimal amount of information that is equivalent to the requirement that
a function preserves $\moment([0,1])$ or $\moment([-1,1])$ when applied
entrywise. We will demonstrate that requiring a function to preserve
moments for measures supported on at most three points, is equivalent to
preserving moments for all measures. In particular, this shows that
preserving positivity for positive Hankel matrices of rank at most three
implies positivity preservation for all positive matrices.

This latter point prompts a comparison to the case of Toeplitz matrices
considered in~\cite{Rudin59}. Rudin proved that
Theorem~\ref{TSchoenberg}(3) holds if $F$ preserves positivity on a
two-parameter family of Toeplitz matrices with rank at most $3$, namely
\begin{equation}\label{Etoeplitz}
\{ ( a + b \cos ((i-j) \theta) )_{i,j \geq 1} : \
a,b \geq 0, \ a+b < 1 \},
\end{equation}
where $\theta$ is a fixed real number such that $\theta / \pi$ is
irrational. Similarly, the present work shows that for power moments,
it suffices to work with families of positive Hankel matrices of rank
at most three. Theorem~\ref{T2sided-general}(1) contains the precise
details.

\section{Moment transformers on $[0,1]$}\label{Smoments01}

Over the course of the next four sections, we will formulate and prove
strengthened versions of the announced results.

Here, we provide two proofs of Theorem~\ref{T1sided}. The first is
natural from the point of view of moments and Hankel matrices.  The
proof proceeds by first deriving from positivity considerations some
inequalities satisfied by all moments transformers. We then obtain the
desired characterization by appealing to classical results on
completely monotonic functions.  This is in the spirit of Lorch and
Newman \cite{Lorch-Newman}, who in turn are very much indebted to the
original Hausdorff approach to the moment problem via summation rules
and higher-order finite differences.

Using Theorem~\ref{Tbernstein}, we now provide our first proof
of Theorem \ref{T1sided}.

\begin{proof}[Proof~1 of Theorem~\ref{T1sided}]
The `if' part follows from two statements: (i) absolutely monotonic
entire functions preserve positivity on all matrices of all orders,
by the Schur product theorem; (ii) moment matrices from elements of
$\moment([0,1])$ have zero entries if and only if
$\mu = a \delta_0$ for some $a \geq 0$.

Conversely, suppose the function $F$ preserves $\moment([0,1])$ when
applied entrywise, i.e., given any~$\mu \in \meas^+([0,1])$, there exists
$\sigma \in \meas^+([0,1])$ such that
\[
F(s_k(\mu)) = s_k(\sigma) \quad \text{for all } k \geq 0.
\]
Let $p(t) = a_0 t^0 + \cdots + a_d t^d$ be a real polynomial such that
$p(t) \geq 0$ on $[0,1]$. Then,
\begin{equation}\label{Etrick}
0 \leq \int_0^1 p(t) \std\sigma(t) = \sum_{k=0}^d a_k s_k(\sigma)
= \sum_{k=0}^d a_k F(s_k(\mu)).
\end{equation}

Here and below, we employ~(\ref{Etrick}) with a careful choice of
measure~$\mu$ and polynomial~$p$ to deduce additional information
about the function $F$. In the present situation, fix finitely many
scalars $c_j$, $t_j > 0$ and an integer $n \geq 0$, and set
\begin{equation}\label{Echoice}
p(t) = (1-t)^n \quad \text{and} \quad
\mu = \sum_j e^{- t_j \alpha} c_j \delta_{e^{-t_j h}},
\end{equation}
where $\alpha > 0$ and $h>0$. Now let $g(x) := \sum_j c_j e^{-t_j x}$,
and apply~(\ref{Etrick}) to see that the forward finite differences of
$F \circ g$ alternate in sign. That is,
\[
\sum_{k=0}^n (-1)^k \binom{n}{k} F\left(
\sum_j c_j e^{- t_j \alpha - t_j kh}\right) \geq 0,
\]
so $(-1)^n \Delta^n_h(F \circ g)(\alpha) \geq 0$. As this holds
for all $\alpha$, $h > 0$ and all $n \geq 0$, it follows that
$F \circ g : (0,\infty) \to (0,\infty)$ is completely monotonic for
all~$\mu$ as in~(\ref{Echoice}). Using the weak density of such
measures in $\meas^+((0,\infty))$, together with Bernstein's theorem
(Theorem \ref{Tbernstein}), it follows that
$F \circ g$ is completely monotonic on $(0,\infty)$ for all completely
monotonic functions $g : (0,\infty) \to (0,\infty)$. Finally, a
theorem of Lorch and Newman \cite[Theorem 5]{Lorch-Newman} now gives
that $F : (0,\infty) \to (0,\infty)$ is absolutely monotonic.
\end{proof}

Our second proof of Theorem \ref{T1sided} involves a significant
relaxation of its hypotheses. Our first observation is that, if $F$
preserves positivity for $2 \times 2$ matrices, and sends
$\moment(\{ 1, u_0 \})$ to $\moment(\R)$ for a single $u_0 \in (0,1)$,
then $F$ is absolutely monotonic on $(0,\infty)$. Further relaxation
may be obtained by working with mass-constrained measures.

\begin{theorem}\label{T1sided-general}
Fix scalars $\rho$ and $u_0$, with $0 < \rho \leq \infty$ and
$u_0 \in (0,1)$. Given a function $F : [0,\rho) \to \R$, the following
are equivalent.
\begin{enumerate}
\item The map $F[-]$ sends
$\momentr(\{ 1,u_0 \}) \cup \momentr(\{ 0, 1 \})$ into
$\moment(\R)$, and
$F(a) F(b) \geq F(\sqrt{ab})^2$ for all~$a$, $b \in [0, \rho)$.

\item The map $F[-]$ sends $\momentr([0,1])$ into
$\moment([0,1])$.

\item The function $F$ agrees on $(0,\rho)$ with an absolutely
monotonic entire function and
$0 \leq F(0) \leq \lim_{\epsilon \to 0^+} F(\epsilon)$.
\end{enumerate}
If $F$ is known to be continuous on $(0,\rho)$, then the second
hypothesis in~$(1)$ may be omitted.
\end{theorem}

Note that assertion (1) is \emph{a priori} significantly weaker than
the requirement that $F$ preserves $\moment([0,1])$, at least when
$\rho = \infty$, say. Moreover, hypothesis~(3) here is the same as
hypothesis~(4) in Theorem~\ref{Ttn}, and Theorem~\ref{T1sided-general}
is used to prove that result in Section~\ref{TN}.

We now turn to proving Theorem~\ref{T1sided-general}. This requires
results on functions preserving positivity for matrices of a fixed
dimension, which we now develop.

As shown in \cite[Theorem 4.1]{GKR-lowrank}, the same result can be
obtained by working only with a particular family of rank-two
matrices, without the continuity assumption, and on any domain
$(0,\rho)$ as above. In the next theorem, Horn's hypotheses are
relaxed even further by making appeal only to Hankel matrices.

\begin{theorem}\label{Thorn-hankel}
Let $F : I \to \R$, where $I := ( 0, \rho )$ and
$0 < \rho \leq \infty$. Fix $u_0 \in (0,1)$ and an integer $N \geq 3$,
and let $\bu := (1, u_0, \dots, u_0^{N-1})^T$. Suppose $F[-]$
preserves positivity on $\bp_2(I)$, and
$F[ A ] \in \bp_N( \R )$ for the family of Hankel matrices
\begin{equation}\label{Ehorn}
\{ A = a \bone_{N \times N} + b \bu \bu^T :\ a \in [ 0, \rho ), \ 
b \in [ 0, \rho - a), \ 0 < a+b < \rho \}.
\end{equation}
Then $F \in C^{N - 3}(I)$, with
\[
F^{( k )}( x ) \geq 0 \qquad
\text{for all } x \in I \qquad (0 \leq k \leq N - 3),
\]
and $F^{(N - 3)}$ is a convex non-decreasing function on~$I$.
If, further, $F \in C^{N - 1}( I )$, then $F^{( k )}( x ) \geq 0$ for
all $x \in I$ and $0 \leq k \leq N - 1$.

Finally, if $F$ is assumed to be continuous on $I$, then the
assumption that $F$ preserves positivity on $\bp_2(I)$ is not
necessary.
\end{theorem}

\begin{remark}\label{Rhorn-hankel}
In fact, our proof of Theorem~\ref{Thorn-hankel} reveals that these
hypotheses may be relaxed slightly, by replacing the test set
$\bp_2((0,\rho))$ with the collection of rank-one matrices
$\bp_2^1((0,\rho))$ and all matrices of the form
\begin{equation}\label{Eabbb}
\begin{pmatrix} a & b \\ b & b \end{pmatrix}
\qquad \text{with } a > b > 0.
\end{equation}
\end{remark}

The proof of  Theorem~\ref{Thorn-hankel} relies on Lemma ~\ref{Lrank1Hankel}.

\begin{proof}[Proof of Theorem \ref{Thorn-hankel}]
If $F \in C(I)$, then the result follows by repeating the argument
in~\cite[Theorem~1.2]{horn}, but with the vector $\alpha$
replaced by a vector $\bu \in \R^N$ as in Lemma~\ref{Lrank1Hankel}.

Now suppose $F$ is an arbitrary function, which is not identically
zero on~$(0,\rho)$; we claim that $F$ must be continuous.  We first
show that $F(x) \neq 0$ for all $x \in (0,\rho)$. Indeed, suppose
$F(c) = 0$ for some $c \in (0,\rho)$. Given $d \in (c, \rho)$, define
a sufficiently long geometric progression $u'_0 = c$, \dots,
$u'_n = d$, such that $u'_{n+1} \in (d,\rho)$. By considering the
matrices
\[
F[A_j], \quad \text{where } A_j := \begin{pmatrix} u'_j & u'_{j+1}\\
u'_{j+1} & u'_{j+2} \end{pmatrix}, \qquad 0 \leq j \leq n-1,
\]
we obtain that $F(d) = 0$ for all $d \in (c,\rho)$. A similar argument
applies to $d \in (0,c)$, showing that $F \equiv 0$ on $(0,\rho)$.

Next, since $F[-]$ preserves positivity on $\bp_2^1((0,\rho))$ and is
positive on $(0,\rho)$, it follows that $g :  x \mapsto \log F(e^x)$ is
midpoint convex on the interval $(-\infty,\log \rho)$.  Moreover,
applying $F[-]$ to matrices of the form (\ref{Eabbb})
shows that $F$ is non-decreasing. Hence, by
\cite[Theorem~71.C]{roberts-varberg}, the function $g$ is necessarily
continuous on~$(-\infty,\log \rho)$, and so~$F$ is continuous
on~$(0,\rho)$. This proves the result in the general case.
\end{proof}

Using the above result, we can now prove Theorem~\ref{Thorn-hankel}.

\begin{proof}[Proof of Theorem \ref{Thorn-hankel}]
If $F \in C(I)$, then the result follows by repeating the argument
in~\cite[Theorem~1.2]{horn}, but with the vector $\alpha$
replaced by a vector $\bu \in \R^N$ as in Lemma~\ref{Lrank1Hankel}.

Now suppose $F$ is an arbitrary function, which is not identically
zero on~$(0,\rho)$; we claim that $F$ must be continuous.  We first
show that $F(x) \neq 0$ for all $x \in (0,\rho)$. Indeed, suppose
$F(c) = 0$ for some $c \in (0,\rho)$. Given $d \in (c, \rho)$, define
a sufficiently long geometric progression $u'_0 = c$, \dots,
$u'_n = d$, such that $u'_{n+1} \in (d,\rho)$. By considering the
matrices
\[
F[A_j], \quad \text{where } A_j := \begin{pmatrix} u'_j & u'_{j+1}\\
u'_{j+1} & u'_{j+2} \end{pmatrix}, \qquad 0 \leq j \leq n-1,
\]
we obtain that $F(d) = 0$ for all $d \in (c,\rho)$. A similar argument
applies to $d \in (0,c)$, showing that $F \equiv 0$ on $(0,\rho)$.

Next, since $F[-]$ preserves positivity on $\bp_2^1((0,\rho))$ and is
positive on $(0,\rho)$, it follows that $g :  x \mapsto \log F(e^x)$ is
midpoint convex on the interval $(-\infty,\log \rho)$.  Moreover,
applying $F[-]$ to matrices of the form (\ref{Eabbb})
shows that $F$ is non-decreasing. Hence, by
\cite[Theorem~71.C]{roberts-varberg}, the function $g$ is necessarily
continuous on~$(-\infty,\log \rho)$, and so~$F$ is continuous
on~$(0,\rho)$. This proves the result in the general case.
\end{proof}

Finally, we turn to the proof of Theorem~\ref{T1sided-general}, which
provides a second proof of Theorem~\ref{T1sided} which is more
informative. We first observe that Theorem~\ref{Thorn-hankel} can be
reformulated in terms of moment sequences, using the fact that the
matrices occurring in the statement of the theorem can be realized as
truncations of positive Hankel matrices; see
Definition~\ref{Dtruncate}.

\begin{theorem}\label{Treformulation}
Let $F : I \to \R$, where $I = (0,\rho)$ and
$0 < \rho \leq \infty$, and fix $N \geq 3$. Suppose
$F[-]$ maps the moment sequences in $\momentr_{2N - 2}( \{ 1, u_0 \} )$
with positive entries to
\[
\{ ( s_0( \mu ), \ldots, s_{2 N - 3}( \mu ), s_{2 N - 2}( \mu ) + t )
: \mu \in \meas^+( \R ), \ t \geq 0 \}
\]
for some $u_0 \in (0,1)$, and the moment sequences in
$\momentr_2( \{ 0, 1 \} ) \cup \momentr_2( \{ u \} )$ with positive
entries to $\moment_2(\R)$ for all $u \in (0,1)$. Then
$F \in C^{N - 3}(I)$, with
\[
F^{( k )}( x ) \geq 0 \qquad
\text{for all } x \in I \qquad (0 \leq k \leq N - 3),
\]
and $F^{(N - 3)}$ is a convex non-decreasing function on~$I$.
If, further, it is known that $F \in C^{N - 1}(I)$, then
$F^{( k )}( x ) \geq 0$ for all $x > 0$ and $0 \leq k \leq N - 1$.

If $F$ is continuous on~$I$, then the assumption that $F[-]$ maps
elements of $\momentr_2(\{ u \})$ into $\moment_2(\R)$ for
all~$u \in (0,1)$ may be omitted.
\end{theorem}

\begin{proof}
In view of Hamburger's Theorem for truncated moment sequences, a
Hankel matrix with entries in the first and last columns given by
\[
s_0, \dots, s_{N-1} \quad \text{and} \quad
s_{N-1}, \dots, s_{2N-2}
\]
is positive if and only if
$( s_0, \dots, s_{2N-3} ) \in \moment_{2N-3}(\R)$, and
$s_{2N-2} \geq \int x^{2N-2} \std\mu$, where~$\mu$ is any non-negative
measure  with the first $2N-2$ moments equal to
$( s_0, \dots, s_{2N-3} )$. (For details, see
Akhiezer's book \cite[Theorem~2.6.3]{Akhiezer}.)

Furthermore, in order to show continuity in Theorem~\ref{Thorn-hankel}
we only required $2 \times 2$ submatrices, of the form~(\ref{Eabbb})
or of rank one. Moreover, every matrix in $\bp_2(\R)$ is a truncated
moment matrix.

These observations show that Theorem~\ref{Treformulation} is
equivalent to Theorem~\ref{Thorn-hankel}.
\end{proof}

We now prove Theorem~\ref{T1sided-general}, with the help of
Theorem~\ref{Treformulation}. 

\begin{proof}[Proof of Theorem \ref{T1sided-general}]
Clearly $(2) \implies (1)$. Next, assume (3) holds, and suppose
$\mu \in \meas^+([0,1])$ with $s_0(\mu) < \rho$.  If
$\mu = a \delta_0$ for some $a \geq 0$ then~$(2)$ is immediate;
henceforth we will assume $H_\mu$ has no zero entries, where $H_\mu$
is as defined in~(\ref{Ehankel}). Now, $F[H_\mu]$ is positive, by the
Schur product theorem and the fact that the only moment matrices
arising from elements of $\momentr([0,1])$ with zero entries come from
$\momentr(\{ 0 \})$. Clearly $F[\bs(\mu)]$ is uniformly bounded, hence
comes from a unique measure $\sigma$ supported on $[-1,1]$, by
Theorem~\ref{Thamburger}. Recalling Definition~\ref{Dtruncate}, we
have that
\[
F[H_\mu]^{(1)} = \sum_{n \geq 0} c_n [H_\mu^{(1)}]^{\circ n},
\]
where $F(x) = \sum_{n \geq 0} c_n x^n$ by the hypotheses and
Theorem~\ref{Twidder}. Note that $F[H_\mu]^{(1)}$ is positive, by the
above computation and Theorem~\ref{Thamburger}, since $\mu$ is
supported on~$[0,1]$. By the same result, $\sigma \in \meas^+([0,1])$,
which gives~(2).

It remains to show $(1) \implies (3)$. It is immediate that
mapping $\momentr_2(\{ 0, 1 \})$ into $\moment_2(\R)$ is equivalent to
mapping $\momentr(\{ 0, 1 \})$ into $\moment(\R)$. Thus, by
Theorem~\ref{Treformulation}, it holds that $F^{(k)}(x) \geq 0$ for
all~$x > 0$ and all~$k \geq 0$. Theorem~\ref{Twidder} now gives the
result, apart from the assertion about~$F(0)$, but this is immediate.
\end{proof}

We conclude this part by explaining why Theorem~\ref{T1sided-general}
provides a minimal set of rank-constrained positive semidefinite
matrices for which positivity preservation is equivalent to absolute
monotonicity.

\begin{definition}
For $1 \leq k \leq N$, let $\bp_N^k(I)$ denote the matrices in
$\bp_N(I)$ of rank at most $k$.
\end{definition}

\begin{remark}
A smaller set of rank-constrained matrices than that employed for
Theorem~\ref{T1sided-general} could not include a sequence of matrices
in $\bigcup_{N = 1}^\infty \bp_N^2([0,\rho))$ of unbounded dimension,
hence would be contained in
$P'_N := \bigcup_{n = 1}^N \bp_n^2( [ 0, \rho ) ) \cup
\bigcup_{n = 1}^\infty \bp_n^1( [ 0, \rho ) )$
for some $N \geq 1$.  However, as noted in the paragraphs preceding
Proposition~\ref{Pjain} below, the map $x \mapsto x^\alpha$ preserves
positivity on $P'_N$ for all $\alpha \geq N-2$, and such a function
may be non-analytic.
\end{remark}

\begin{remark}
The proof of Theorem~\ref{T1sided-general} also strengthens a 1979
result of Vasudeva~\cite{vasudeva79}. Vasudeva showed for
$I = ( 0, \infty )$ that if $F : I \to \R$ preserves positivity
entrywise on $\bp_N(I)$ for all $N \geq 1$ then $F$ is absolutely
monotonic and so is represented by a convergent power series on
$I$. The proof above shows that Vasudeva's result also holds if $I$ is
replaced by $(0,\rho)$ for any $\rho > 0$ and, for every $N$, the set
$\bp_N(I)$ is replaced by the subset of Hankel matrices within it of
rank at most $2$.
\end{remark}

\subsection{Hankel-matrix positivity preservers in fixed
dimension}\label{S31}

We conclude this section by addressing briefly the fixed-dimension
case for powers and analytic functions, as studied by FitzGerald
and Horn, and also in previous work by the authors
\cite{BGKP-fixeddim, FitzHorn, GKR-crit-2sided}. Our first result
shows that considerations of Hankel matrices may be used to strengthen
the main result in~\cite{BGKP-fixeddim}.
 
\begin{theorem}\label{Tthreshold2}
Fix $\rho > 0$ and integers $N \geq 1$ and $M \geq 0$, and let
$F( z ) = \sum_{j = 0}^{N - 1} c_j z^j + c' z^M$ be a
polynomial with real coefficients. The following are equivalent.
\begin{enumerate}
\item $F[ - ]$ preserves positivity on
$\bp_N( \overline{D}(0, \rho ) )$, where $\overline{D}( 0, \rho )$ is
the closed disc in the complex plane with center~$0$ and radius~$\rho$.

\item The coefficients $c_j$ satisfy either
$c_0$, \ldots, $c_{N - 1}$, $c' \geq 0$, or
$c_0$, \ldots, $c_{N - 1} > 0$ and
$c' \geq -\mathcal{C}( \bc; z^M; N, \rho )^{-1}$, where
\[
\mathcal{C}( \bc; z^M; N, \rho ) :=
\sum_{j = 0}^{N - 1} \binom{M}{j}^2 \binom{M - j - 1}{N - j - 1}^2
\frac{\rho^{M - j}}{c_j}.
\]

\item $F[ - ]$ preserves positivity on
Hankel matrices in $\bp_N^1( ( 0, \rho ) )$.
\end{enumerate}
\end{theorem}

The strengthening here is the addition of the word `Hankel' to
hypothesis~(3).

\begin{remark}\label{Rspecialhankel}
As the following proof of Theorem~\ref{Tthreshold2} shows,
assumption~(3) can be relaxed further, by assuming $F$ preserves
positivity on a distinguished family of Hankel matrices. More
precisely, it can be replaced by
\begin{enumerate}
\item[$(3')$]
$F[-]$ preserves positivity on two sequences of rank-one Hankel
matrices,
\[
\{ b^n \rho \bu(b) \bu(b)^T,\ \rho \bu(b^n) \bu(b^n)^T : 
n \geq 1 \}, \quad \text{ for any fixed } b \in (0,1), 
\]
where
\begin{equation}\label{Euvector}
\bu(\epsilon) := (1 - \epsilon, (1-\epsilon)^2, \dots, (1-\epsilon)^N)^T,
\qquad \text{for any } \epsilon \in (0,1).
\end{equation}
\end{enumerate}
Note that $\bu(\epsilon) \bu(\epsilon)^T \in \bp_N^1(\R)$
is Hankel, by Lemma~\ref{Lrank1Hankel}.
\end{remark}

Thus, Remark~\ref{Rspecialhankel} gives a notable reduction of the
$N$-dimensional parameter space, $\bp_N^1((0,\rho))$, to the countable
subset of Hankel matrices required in~($3'$). If $N>1$, this is indeed
minimal information required to derive Theorem~\ref{Tthreshold2}(2),
since the extreme critical value $\mathcal{C}( \bc; z^M; N, \rho )$
cannot be attained on any finite set of matrices in
$\bp_N^1((0,\rho))$.

As a first step towards the proof of Theorem~\ref{Tthreshold2}, we
recall from \cite[Lemma~2.4]{BGKP-fixeddim} that, under suitable
differentiability assumptions, the conclusions of
Theorem~\ref{Thorn-hankel} still hold if one considers only rank-one
matrices. We now formulate a slightly stronger version of this result.

\begin{proposition}\label{Phorn-hankel}
Let $F \in C^\infty( (-\rho, \rho ))$, where $0 < \rho \leq \infty$. Fix
a vector $\bu \in (0,\sqrt{\rho})^N$ with distinct coordinates, and
suppose $F[b_n \bu \bu^T] \in \bp_N(\R)$ for a positive real
sequence $b_n \to 0^+$. Then the first $N$ non-zero derivatives of $F$
at $0$ are strictly positive.
\end{proposition}

The assumptions and conclusions of this result are similar to those of
Theorem~\ref{Thorn-hankel} above; a common generalization of both
results can be found in~\cite{Kh}.

\begin{proof}
For ease of exposition, we will assume $F$ has at least $N$ non-zero
derivatives at $0$, say of orders $m_1 < \cdots < m_N$, where
$m_1 \geq 0$. By results on generalized Vandermonde determinants
\cite[Chapter~XIII, \S8, Example~1]{Gantmacher_Vol2}, the vectors
$\{ \bu^{\circ m_j} : 1 \leq j \leq N \}$ are linearly independent.
Now, by Taylor's theorem, 
\begin{equation}\label{EhankelTaylor}
F[ b_n \bu \bu^T] = \sum_{j=1}^N \frac{F^{(m_j)}(0)}{m_j!}
b_n^{m_j} \bu^{\circ m_j} (\bu^{\circ m_j})^T + o( b_n^{m_N} ).
\end{equation}
For each $1 \leq k \leq N$, choose $\bv_k \in \R^N$ such that
$\bv_k^T \bu^{\circ m_j} = \delta_{j,k} m_j!$. Then,
\[
b_n^{-m_k} \bv_k^T F[ b_n \bu \bu^T ] \bv_k =
m_k! F^{(m_k)}(0) + o( b_n^{m_N - m_k} ) \geq 0,
\]
and letting $n \to \infty$ concludes the proof.
\end{proof}

We now use Proposition~\ref{Phorn-hankel} to prove the theorem.

\begin{proof}[Proof of Theorem \ref{Tthreshold2}]
In view of Remark~\ref{Rspecialhankel}
and~\cite[Theorem~1.1]{BGKP-fixeddim}, it suffices to show that
$(3') \implies (2)$.

Assume $(3')$ holds, and consider first the sequence
$b^n \rho \bu(b) \bu(b)^T$. If $0 \leq M < N$, then the result follows
from Proposition~\ref{Phorn-hankel}, since the critical value
is precisely $\mathcal{C}( \bc; z^M; N, \rho ) = c_M^{-1}$. Now suppose
$M \geq N$. Again using Proposition~\ref{Phorn-hankel}, either
$c_0$, \dots, $c_{N-1}$ and $c'$ are all non-negative, or else we have
that $c_0$, \dots, $c_{N-1} > 0 > c_M$.

In the latter case, to prove that
$c_M \geq -\mathcal{C}( \bc; z^M; N, \rho )^{-1}$, we use
the sequence $\rho \bu(b^n) \bu(b^n)^T$, where $\bu(b^n)$ is
defined as in~(\ref{Euvector}). Let $\bu_n := \sqrt{\rho} \bu(b^n)$ for
$n \geq 1$. Then \cite[Equation~(3.11)]{BGKP-fixeddim} implies that
$0 \leq \det |c_M|^{-1} F[ \bu_n \bu_n^T ]$, and so
\[
|c_M|^{-1} \geq \sum_{j=0}^{N-1} \frac{s_{\mu(M,N,j)}(\bu_n)^2}{c_j},
\]
where $\mu(M,N,j)$ is the hook partition
$(M-N+1,1,\dots,1,0,\dots,0)$, with $N-j-1$ ones after the first entry
and then $j$ zeros, and $s_{\mu(M,N,j)}$ is the corresponding Schur
polynomial.  As $n \to \infty$, so
$\bu_n \to \sqrt{\rho}(1,\dots,1)^T$.
The Weyl Character Formula in type~$A$ gives that
$\displaystyle s_{\mu(M,N,j)}(1,\dots,1) = 
\binom{M}{j} \binom{M-j-1}{N-j-1}$,
and it follows that
\[
|c_M|^{-1} \geq \sum_{j=0}^{N-1} \binom{M}{j}^2 \binom{M-j-1}{N-j-1}^2
\frac{\rho^{M-j}}{c_j} = \mathcal{C}( \bc; z^M; N, \rho ).
\]
Thus (2) holds, and this concludes the proof.
\end{proof}

Finally, we consider the question of which real powers preserve
positivity on $N \times N$ Hankel matrices. Recall that the Schur
product theorem guarantees that integer powers $x \mapsto x^k$
preserve positivity on~$\bp_N((0,\infty))$. It is natural to ask if
any other real powers do so. In~\cite{FitzHorn}, FitzGerald and Horn
solved this problem, and uncovered an intriguing transition.
In their main result, they show that the power function
$x \mapsto x^\alpha$ preserves positivity entrywise on
$\bp_N((0,\infty))$ if and only if $\alpha$ is a non-negative integer
or $\alpha \geq N-2$. The value $N-2$ is known in the literature as
the \emph{critical exponent} for preserving positivity.

As shown in \cite{GKR-crit-2sided}, the critical exponent remains
unchanged upon restricting the problem to preserving positivity on
$\bp_N^k((0,\infty))$ for any $k\geq2$. More precisely, for each
non-integral $\alpha \in (0,N-2)$, there exists a rank-two matrix
$A \in \bp_N^2((0,\infty))$ such that
$A^{\circ \alpha} \not\in \bp_N$; see~\cite{GKR-crit-2sided} for more
details. 

As we now show, the result does not change when restricted to the set
of positive semidefinite Hankel matrices.

\begin{proposition}\label{Pjain}
Let $2 \leq k \leq N$ and let $\alpha \in \R$. The following are
equivalent.
\begin{enumerate}
\item The power function $x \mapsto x^\alpha$ preserves positivity
when applied entrywise to Hankel matrices in $\bp_N^k((0,\infty))$.

\item The power $\alpha$ is a non-negative integer or
$\alpha \geq N - 2$.
\end{enumerate} 
Moreover, there exists a Hankel matrix $A \in \bp_N^2((0,\infty))$
such that $A^{\circ \alpha} \not\in \bp_N$ for all non-integral
$\alpha \in (0,N-2)$.
\end{proposition}

\begin{proof}
By the main result in~\cite{Jain}, for pairwise distinct real numbers
$x_1, \dots, x_N > 0$, the matrix $((1+x_i x_j)^\alpha)_{i,j=1}^N$ is
positive semidefinite if and only if $\alpha$ is a non-negative
integer or $\alpha \geq N-2$. The result now follows
immediately, by Lemma \ref{Lrank1Hankel}.
\end{proof}

Note that replacing $( 0, \infty )$ with $( 0, \rho )$ for some $\rho$
with $0 < \rho < \infty$ leads to the same classification of entrywise
powers preserving positivity on the reduced test set.

\section{Totally non-negative matrices}\label{TN}

With a better understanding of the endomorphisms of moment sequences
of positive measures, we turn next to the structure of preservers of
total non-negativity, in both the fixed-dimension and dimension-free
settings. Recall that a rectangular matrix is \emph{totally non-negative}
if every minor is a non-negative real number.

We begin with the well-known fact that moment sequences of positive
measures on~$[ 0, \infty )$ are in natural correspondence with totally
non-negative Hankel matrices.

\begin{lemma}\label{Ltn}
A real sequence $(s_k)_{k=0}^\infty$ is the moment sequence of a
positive measure on~$[ 0, \infty )$ if and only if the corresponding
semi-infinite Hankel matrix $H := (s_{i+j})_{i, j = 0}^\infty$ is
totally non-negative. The measure is supported on $[ 0, 1 ]$ if and
only if the entries of~$H$ are uniformly bounded.
\end{lemma}

\begin{proof}
The first claim is a consequence of well-known results in
the theory of moments \cite{GK,STmoment}, as outlined in the
introduction to \cite{FJS}. 
For measures on $[0,1]$, the result now follows via
Theorem~\ref{Thamburger}(3).
\end{proof}

Lemma \ref{Ltn} also has a finite-dimensional version, which will be
required in the proof of Theorem~\ref{Ttn}.

\begin{lemma}[{\cite[Corollary 3.5]{FJS}}]\label{Ltn_fin}
Let $A$ be an $N \times N$ Hankel matrix. Then $A$ is totally
non-negative if and only if $A$ and its truncation $A^{(1)}$ have
non-negative principal minors.
\end{lemma}

With Lemmas~\ref{Ltn} and~\ref{Ltn_fin} in hand, we can now establish
our characterization of positivity preservers on~$\HTN$.

\begin{proof}[Proof of Theorem~\ref{Ttn}]
Suppose $(1)$ holds, and let $A \in \HTN$. Then both $F[A]$ and
$F[A]^{(1)} = F[A^{(1)}]$ have non-negative principal minors, 
so $F[A] \in \HTN$, by Lemma~\ref{Ltn_fin}. Thus $(1) \implies (2)$.

That $(2) \implies (3)$ follows directly from Lemma~\ref{Ltn}.
Next, suppose $(3)$ holds and let~$a > 0$ and $b \geq 0$. Applying
$F[-]$ to the first few moments of the measure $a \delta_{\sqrt{b/a}}$
shows that $F(a) F(b) \geq F( \sqrt{a b } )^2$. By
Theorem~\ref{T1sided-general}, we conclude that $(4)$ holds.

Finally, suppose $(4)$ holds and let $H \in \HTN_N$ for some
$N \geq 1$. If every entry of $H$ is non-zero, then $F[H]$ is positive
semidefinite, by the Schur product theorem. Otherwise, suppose $H$ has
a zero entry. Denote the entries in the first row and last column of
$H$ by $s_0$, \dots, $s_{N-1}$ and $s_{N-1}$, \dots, $s_{2N-2}$,
respectively. By considering $2 \times 2$ minors, it is easy to show
that
\[
s_0 = 0 \implies s_1 = 0 \iff s_2 = 0 \iff \dots \iff s_{2N-3} = 0
\impliedby s_{2N-2} = 0.
\]
Consequently, if $(4)$ holds and an entry of $H$ is zero, then
$F[H] \in \bp_N$.
\end{proof}

\begin{remark}
While Theorem~\ref{Ttn} is more natural to state for functions with
domain $[ 0, \infty )$, the proof goes through verbatim for
$F : [ 0, \rho ) \to \R$, where $0 < \rho < \infty$. In this
case, the test set $\mathcal{H}^{++}$ in the first two assertions of
Theorem~\ref{Ttn} (but not the target set) must be replaced by its
subset of matrices with entries in $[ 0, \rho )$.
\end{remark}

Next we examine the class of polynomial maps that,
when applied entrywise, preserve total non-negativity for Hankel
matrices of a fixed dimension. First, note that the analogue of the
Schur product theorem holds for totally non-negative Hankel matrices
\cite[Theorem~4.5]{FJS}; this also follows from Lemma~\ref{Ltn_fin}.
Second, note that the Hankel matrix
$H_\epsilon :=\bu(\epsilon) \bu(\epsilon)^T$ is totally non-negative
for all~$\epsilon \in ( 0, 1 )$, where $\bu(\epsilon)$ was defined
in~(\ref{Euvector}):
\[
\bu(\epsilon) := ( 1 - \epsilon, \dots, ( 1 - \epsilon )^N )^T.
\]
This holds because the elements of $H_\epsilon$ are all positive, and
the $k \times k$ minors of $H_\epsilon$ vanish if $k \geq 2$. As a
consequence, Proposition~\ref{Phorn-hankel} implies that if $F$ is a
polynomial which preserves positive semidefiniteness on $\HTN_N$, then
the first $N$ non-zero coefficients of $F$ must be positive.

The following result shows that the next coefficient can be
negative, with the same threshold as in Theorem~\ref{Tthreshold2}.

\begin{theorem}
Let $\rho$, $N$, $M$ and
$F( z ) = \sum_{j = 0}^{N - 1} c_j z^j + c' z^M$
be as in Theorem \ref{Tthreshold2}. The following are equivalent.
\begin{enumerate}
\item $F[ - ]$ preserves total non-negativity for elements of $\HTN_N$
with entries in $[ 0, \rho )$.

\item The coefficients $c_j$ satisfy either
$c_0$, \ldots, $c_{N - 1}$, $c' \geq 0$, or
$c_0$, \ldots, $c_{N - 1} > 0$ and
$c' \geq -\mathcal{C}( \bc; z^M; N, \rho )^{-1}$, where
\[
\mathcal{C}( \bc; z^M; N, \rho ) :=
\sum_{j = 0}^{N - 1} \binom{M}{j}^2 \binom{M - j - 1}{N - j - 1}^2
\frac{\rho^{M - j}}{c_j}.
\]

\item $F[ - ]$ preserves positivity for rank-one
elements of $\HTN_N$ with entries in $( 0, \rho )$.

\end{enumerate}
\end{theorem}

\begin{proof}
Clearly $(1) \implies (3)$, and  $(3) \implies (3')$, where the
assertion $(3')$ is as in Remark~\ref{Rspecialhankel}. That
$(3') \implies (2)$ follows from the proof of
Theorem~\ref{Tthreshold2}.

To prove $(2) \implies (1)$, first observe from
Theorem~\ref{Tthreshold2} that $F[-]$ preserves positivity on
$\bp_N( [ 0, \rho ] )$. Given any matrix
$A \in \HTN_N$ with entries in $[ 0, \rho )$, let $B$ denote the
$N \times N$ matrix obtained by deleting the first column and last row
of $A$, and then adding a last row and column of zeros. Both $A$ and
$B$ are positive semidefinite, and therefore so are~$F[ A ]$ and
$F[ B ]$. Hence $F[A]$ and $F[A]^{(1)} = F[A^{(1)}]$ have non-negative
principal minors, since the principal minors of the latter are
included in those of $F[ B ]$. Lemma~\ref{Ltn_fin} now gives the
result.
\end{proof}

It is trivial that the Hadamard (entrywise) power $H^{\circ \alpha}$
is totally non-negative for all~$H \in \HTN_1 \cup \HTN_2$ if and only
if $\alpha \geq 0$.  For higher dimensions, the situation is as
follows.

\begin{theorem}[{\cite[Theorem~5.11 and Example~5.5]{FJS}}]
\label{T_FJS_crit}
Let $\alpha \in \R$ and $N \geq 2$. The power function~$x^\alpha$
preserves $\HTN_N$ if and only if $\alpha$ is a non-negative integer
or $\alpha \geq N - 2$.
\end{theorem}

Thus the set of powers preserving total non-negativity for Hankel matrices
coincides with the set of powers preserving positivity on
$\bp_N([0,\infty))$, as identified by FitzGerald and Horn
\cite{FitzHorn}.

\begin{remark}
We note that Theorem~\ref{T_FJS_crit} follows from a result of
Jain \cite[Theorem~1.1]{Jain}, since for $x \in ( 0, 1 )$, the
semi-infinite Hankel matrix $(1 + x^{i+j})_{i,j = 0}^\infty$ arises as
the moment matrix of the measure $\delta_1 + \delta_x$, and is
therefore totally non-negative, by Lemma~\ref{Ltn}.
\end{remark}

We conclude this section by examining entrywise preservers of total
non-negativity in the general setting, where the matrices are not
assumed to have a Hankel structure, or to be
symmetric or even square. By Theorem~\ref{Ttn}, every such preserver must
be absolutely monotonic on $( 0, \infty )$. However, it is not
immediately clear how to proceed further with non-symmetric matrices;
the analogue of the Schur product theorem no longer holds in this
situation, as noted in \cite[Example~4.3]{FJS}.

Our next result shows that, when working with rectangular or symmetric
matrices, the set of functions preserving total non-negativity is very
rigid.

\begin{theorem}\label{Ptn}
Suppose $F : [0,\infty) \to \R$. The following are equivalent.
\begin{enumerate}
\item Applied entrywise, the function $F$ preserves total non-negativity
on the set of all rectangular matrices of arbitrary size.

\item Applied entrywise, the function $F$ preserves total non-negativity
on the set of all real symmetric matrices of arbitrary size.

\item The function $F$ is constant or linear. In other words, there
exists $c \geq 0$ such that either $F( x ) \equiv c$, or else
$F( x ) = c x$ for all $x \geq 0$.
\end{enumerate}
\end{theorem}

Contrast this result, especially hypothesis~(2), with
Theorem~\ref{Ttn}.

We defer the proof of Theorem~\ref{Ptn} until we have more closely
examined the case of entire maps. This will give what is
needed to overcome the main technical difficulty in proving
Theorem~\ref{Ptn}.

Recall from \cite[Section 5]{FJS} that if $A$ is a totally non-negative
matrix which is $3 \times 3$, or symmetric and $4 \times 4$, then the
Hadamard power $A^{\circ \alpha}$ is totally non-negative for all $\alpha
\geq N - 2$, where $N$ is the number of rows of $A$.

For larger matrices, very few entire functions preserve
total non-negativity.

\begin{theorem}\label{TpolyTN}
Let $F( x ) = \sum_{n = 0}^\infty c_n x^n$ be entire
with real coefficients. The entrywise map $F[-]$
preserves total non-negativity for $4 \times 4$ matrices if and only if
$F( x ) \equiv c_0$ with $c_0 \geq 0$, or $F( x ) = c_1 x$
for all $x \geq 0$ with $c_1 \geq 0$. The same conclusion
holds if $F[-]$ preserves total non-negativity for symmetric $5 \times 5$
matrices.
\end{theorem}

\begin{proof}
First we consider the $4 \times 4$ case. Note that
one implication is immediate, so suppose $F[-]$ preserves total
non-negativity and is not constant. Let
$A_y := y \, \Id_3 \oplus \mathbf{0}_{1 \times 1}$, where $y \geq 0$
and $\Id_k$ denotes the $k \times k$ identity matrix for $k \geq 1$.
Observing that $F[A_y]$ is totally non-negative, it follows that
$F(y) \geq F(0) \geq 0$ for all $y \geq 0$. If, moreover, $y > 0$ is
such that $F( y ) > F( 0 )$, then from the same observation we
conclude that $F( 0 ) = c_0 = 0$.

Next, suppose for contradiction that
$F( x ) = c_m x^m + O( x^{m + 1} )$, where $m > 1$
and~$c_m \neq 0$. We
make use of the matrix studied in \cite[Example~5.9]{FJS},
\begin{equation}\label{Ealmosthankel}
M := \begin{pmatrix}
3 & 6 & 14 & 36 \\
6 & 14 & 36 & 98 \\
14 & 36 & 98 & 276 \\
36 & 98 & 284 & 842
\end{pmatrix},
\end{equation}
and let $A( x ) := \bone_{4 \times 4} + x M$. By the analysis in
\cite[Example 5.9]{FJS}, the matrix $A( x )$ is totally non-negative
for all $x \geq 0$, while for every real $\alpha > 1$ there exists
$\delta_\alpha > 0$ such that
\[
\det A( x )^{\circ \alpha} < 0 \qquad
\text{for all } x \in ( 0, \delta_\alpha ).
\]
Fix $z \in (0, \delta_m )$, let $t > 0$, and note that
\[
F[ t A( z ) ] = c_m t^m A( z )^{\circ m} + t^{m + 1} C( t, z )
\]
for some $4 \times 4$ matrix $C( t, z )$. Since the matrix on the
left-hand side is totally non-negative, it follows that
\[
0 \leq t^{-4 m} \det F[ t A( z ) ] =
c_m^4 \det A( z )^{\circ m} + O( t ).
\]
Letting  $t \to 0^+$ gives a contradiction. Hence
$c_1 \neq 0$.

Finally, note that
\[
F[ t\, A( x ) ] =
\sum_{n = 1}^\infty c_n t^n ( \bone_{4 \times 4} + x M )^{\circ n} =
\sum_{n = 1}^\infty c_n t^n \sum_{j = 0}^n \binom{n}{j} x^j M^{\circ j} =
\sum_{j = 0}^\infty \beta_j(t) x^j M^{\circ j}, 
\]
where $t \geq 0$ and $\beta_j(t) := \sum_{n = j}^\infty c_n
\binom{n}{j} t^n$. Using a Laplace expansion, it is not hard to see that
\[
\det F[ t\, A(x) ] = \det M_4(t) + O( x^5 ), \quad \text{where} \quad
M_4(t) := \sum_{j = 0}^4 \beta_j(t) x^j M^{\circ j}.
\]
If $R$ is a commutative unital ring containing $x$ and
$\alpha_1$, \ldots, $\alpha_4$ then Appendix~\ref{Ap} gives that
\begin{equation}\label{Emiracle}
\det M_4 =
-57168 \, \alpha_0 \, \alpha_1^2 \, \alpha_2 \, x^4 + O( x^5 ), \qquad
\text{where } M_4 := \sum_{j = 0}^4 \alpha_j x^j M^{\circ j}.
\end{equation}
Taking $R = \R[ t, x ]$
and $\alpha_j = \beta_j(t)$, we have that $M_4$ equals $M_4(t)$.
Since $F[t A(x) ]$ is totally non-negative for all $x \geq 0$,
dividing through by $x^4$ and letting $x \to 0^+$, it follows that
$\beta_0(t) \beta_1(t)^2 \beta_2(t)$ vanishes on an
interval. Since $\beta_j(t) = F^{(j)}(t) / j!$, each $\beta_j$ is also
entire; thus at least one $\beta_j \equiv 0$, whence
$\beta_2(t) \equiv 0$. It follows that $c_n = 0$ for all $n \geq 2$,
as claimed. That $c_1 \geq 0$ now follows by considering $F[\Id_4]$.

This concludes the proof for $4 \times 4$ totally non-negative
matrices. The proof for symmetric $5 \times 5$ matrices now follows,
as \cite[Example~5.10]{FJS} gives a $5 \times 5$ symmetric totally
non-negative matrix containing the matrix $A( x )$ as a $4 \times 4$
minor.
\end{proof}

With this result in hand, we can now complete the outstanding proof in
this section.

\begin{proof}[Proof of Theorem~\ref{Ptn}]
Clearly $(3) \implies (1) \implies (2)$.
Suppose $(2)$ holds. Then, by Theorem~\ref{Ttn}, the function $F$ is
absolutely monotonic on $( 0, \infty )$, and $F( 0 ) \geq 0$.
If $F$ is not constant, then $F( y ) > F( 0 )$ for some $y > 0$. As
$F[ y \, \Id_3 ]$ is totally non-negative, looking at $2 \times 2$
minors now shows that $F( 0 ) = 0$.

To see that $F$ is continuous at $0$, note first that
\[
A :=
\begin{pmatrix} 2 & 1 & 0 \\ 1 & 2 & 1 \\ 0 & 1 & 2 \end{pmatrix}
\]
is totally non-negative. If
$L := \lim_{x \to 0^+} F( x )$, then
$\det F[ t A] \to -L^3 \geq 0$ as $t \to 0^+$, whence
$L = 0$, as desired.

Thus $F$ has the form required to apply Theorem~\ref{TpolyTN}, so
$F( x ) = c_1 x$ for all $x \in [ 0, \infty )$, as required.
\end{proof}

\section{Moment transformers on $[-1,1]$}\label{Smomentsm11}

Equipped with the one-sided result from Theorem \ref{T1sided-general},
we now classify the functions which map the set $\moment([-1,1])$ into
$\moment(\R)$ when applied entrywise. The goal of this section is to
prove the following strengthening of Theorem~\ref{T2sided}, in the
spirit of Theorem~\ref{T1sided-general}.

\begin{theorem}\label{T2sided-general}
Let $F : ( -\rho, \rho ) \to \R$, where $0 < \rho \leq \infty$. The
following are equivalent.
\begin{enumerate}
\item $F[-]$ maps the sequences
$\bigcup_{u \in (0,1)} \momentr(\{ -1, u, 1 \})$ into
$\moment(\R)$.

\item $F[-]$ maps the sequences
$\bigcup_{u \in (0,1)} \momentr(\{ -1, u, 1 \})$ into
$\moment([-1,1])$.

\item $F[-]$ maps $\momentr([-1,1])$ into $\moment(\R)$.

\item $F$ is the restriction to $(-\rho,\rho)$ of an absolutely
monotonic entire function.
\end{enumerate}
\end{theorem}

Recall that Schoenberg and Rudin's result, Theorem~\ref{TSchoenberg},
characterizes positivity preservers for matrices with entries
in~$(-1,1)$. As a consequence of Theorem~\ref{T2sided-general}, we
obtain the following generalization of
Theorem~\ref{TSchoenberg} with a much reduced test set.

\begin{corollary}
The hypotheses of Theorem \ref{TSchoenberg} are equivalent to $F[-]$
preserving positivity on Hankel matrices arising from moment
sequences, with entries in~$( -1, 1 )$ and rank at most
$3$. Furthemore, this theorem holds with $(-1,1)$ and $(0,1)$ replaced
by $(-\rho,\rho)$ and $(0,\rho)$, respectively, whenever
$\rho \in (0,\infty]$.
\end{corollary}

The proof of Theorem~\ref{T2sided-general} requires new ideas, as
previous techniques to prove analogous results are not amenable to the
more general Hankel setting; see Remark~\ref{Rembed}.

As a first step, we obtain the following lemma; together with
Theorem~\ref{Thamburger}, it explains why assertion~(1) in
Theorem~\ref{T2sided-general} can be relaxed to assertion~(2).

Recall the notion of truncated moment sequence from
Definition~\ref{Dmass}.

\begin{lemma}\label{Lbounded}
If $F : (-\rho,\rho) \to \R$ maps entrywise the sequences
$\momentr_2(\{ -1, 1 \})$ into $\moment_2(\R)$, then $F$ is
locally bounded. If $F$ is known to be locally bounded on
$(0,\rho)$, then the set $\momentr_2(\{ -1, 1 \})$
may be replaced by~$\momentr_2(\{ -1 \})$.
\end{lemma}

\begin{proof}
Akin to the proof of Theorem~\ref{Thorn-hankel}, the
assumption implies that~$F$ is non-decreasing, whence
locally bounded, on~$(0,\rho)$. Now let $\mu = a \delta_{-1}$
for any $a \in (0,\rho)$. By considering the leading principal
$2 \times 2$ submatrix of~$F[H_\mu]$, where $H_\mu$ is the
Hankel matrix~(\ref{Ehankel}) associated to the measure~$\mu$,
it follows that $|F(-a)| \leq F(a)$.
\end{proof}

The next step is to use assertion~(2) in Theorem~\ref{T2sided-general}
to establish the continuity of~$F$ on $(-\rho,\rho)$.

\begin{proposition}\label{Pcont}
Fix $v_0 \in (0,1)$ and suppose the function $F: (-\rho,\rho) \to \R$,
where $0 < \rho \leq \infty$, maps entrywise
\[
\momentr_2(\{ -1, 1 \}) \cup \momentr_3(\{ -1, v_0 \}) \cup
\bigcup_{u \in (0,1)} \momentr_4(\{ 1, u \})
\]
into
$\moment_2([-1,1]) \cup \moment_3([-1,1]) \cup \moment_4([-1,1])$.
Then~$F$ is continuous on~$(-\rho,\rho)$.
\end{proposition}

\begin{proof}
As $F$ maps $\momentr_2(\{ -1, 1 \})$ into $\moment_2([-1,1])$,
considering
\[
\mu = \frac{a + b}{2} \delta_1 + \frac{a - b}{2} \delta_{-1}
\quad \text{and} \quad \nu = b \delta_1
\]
shows that $F( a ) \geq F( b ) \geq 0$ whenever
$0 \leq a \leq b < \rho$. It follows immediately that $F$ maps
$\momentr_2(\{ 0, 1 \})$ into $\moment_2(\R)$. Then, by
Theorem~\ref{Treformulation} for~$N=3$ and our assumptions, $F$ is
continuous, non-negative, and non-decreasing on $( 0, \rho )$.
In particular, $F$ has a right-hand limit at $0$, and
\begin{equation}\label{Eplaceholder}
0 \leq F(0) \leq \lim_{\epsilon \to 0^+} F(\epsilon).
\end{equation}

We now fix $v_0 \in (0,1)$ and use the truncated moment sequences in
$\momentr_3(\{ -1, v_0 \})$ to prove two-sided continuity of $F$ at
all points in $(-\rho,0]$. Fix $\beta \in [0,\rho)$, and for $b$ such
that $0 < b < ( \rho - \beta ) / ( 1 + v_0 )$, let
\[
a := \beta + b v_0 \qquad \text{and} \qquad
\mu = a \delta_{-1} + b \delta_{v_0}.
\]
By assumption, we have that
$F[-] : \momentr_3(\{ -1, v_0 \}) \to \moment_3([-1,1])$, so
there exists $\sigma \in \meas^+[-1,1]$ such that
\[
(F(s_0(\mu)), \dots, F(s_3(\mu))) = (s_0(\sigma), \dots, s_3(\sigma)).
\]
If the polynomials $p_\pm(t) := (1 \pm t)(1-t^2)$ then,
\[
\int_{-1}^1 p_{\pm}(t) \std\sigma \geq 0,
\]
since $p_\pm(t)$ are non-negative on $[-1,1]$. Hence (\ref{Etrick})
gives that
\[
F(a+b) - F(a+bv_0^2) \geq
\pm\left( F(-a + bv_0 ) - F(-a + bv_0^3) \right), 
\]
or, equivalently, 
\[
F(\beta+b+b v_0) - F(\beta + b v_0 + bv_0^2) \geq \left|
F(-\beta) - F(-\beta - b (v_0 - v_0^3)) \right|.
\]
Letting $b \to 0^+$ and using the continuity of $F$ on~$(0,\rho)$, we
conclude that $F$ is left continuous at~$-\beta$.  We proceed
similarly to show right continuity of~$F$ at~$-\beta$; let
\[
a := \beta + b v_0^3 \quad \text{and} \quad
\mu = a \delta_{-1} + b \delta_{v_0},
\]
where $b$ is such that $0 < b < ( \rho - \beta ) / ( 1 + v_0^3 )$,
and take $b \to 0^+$ as before.
\end{proof}

\begin{remark}
The integration trick~(\ref{Etrick}) used in the proof of
Proposition~\ref{Pcont} shows that certain linear combinations of
moments are non-negative. The integral it employs may also be
expressed using the quadratic form given by the Hankel moment matrix
for the ambient measure. To see this, suppose $\sigma$ is a
non-negative measure on $[ -1, 1 ]$ with moments of all orders, and
let $H_\sigma := ( s_{j+k}( \sigma ) )_{j,k \geq 0}$ be the associated
Hankel moment matrix. If $f : [ -1, 1 ] \to \R_+$ is continuous then
so its radical $\sqrt{f} : [ -1, 1 ] \to \R_+$, and the latter can be
uniformly approximated on $[ -1, 1 ]$ by a sequence of polynomials
$p_n( t) = \sum_{j = 0}^{d_n} c_{n, j} t^j$.

Thus
\[
\int_{-1}^1 f \std\sigma =
\lim_{n \to \infty} \int_{-1}^1 p_n(t)^2 \std\sigma =
\lim_{n \to \infty} \sum_{j, k \geq 0} c_{n, j} c_{n, k}
\int_{-1}^1 t^{j + k} \std\sigma =
\lim_{n \to \infty} \bv_n^T H_\sigma \bv_n,
\]
where 
\[
\bv_n := ( c_{n,0}, c_{n,1}, \ldots, c_{n,d_n}, 0, 0, \dots)^T
\qquad (n \geq 1).
\]
Now, since the matrix $H_\sigma$ is positive, the limit on the
right-hand side is non-negative and so
$\int_{-1}^1 f \std\sigma \geq 0$.
\end{remark}

With continuity in hand, we can now complete the proof of
Theorem~\ref{T2sided-general}.

\begin{proof}[Proof of Theorem~\ref{T2sided-general}]
Clearly $(4) \implies (3) \implies (1)$ and $(2) \implies (1)$; that
$(1) \implies (2)$ follows from the remarks preceding
Lemma~\ref{Lbounded}. Now suppose $(1)$ holds. By
Proposition~\ref{Pcont}, the function $F$ is continuous on
$(-\rho,\rho)$, so Theorem~\ref{T1sided-general} gives that~$F$ agrees
on~$(0,\rho)$ with a power series~$\widetilde{F}$ having non-negative
Maclaurin coefficients, which is convergent on the disc $D(0,\rho)$.

Let $\mu := a \delta_{-1} + e^x \delta_{e^{-h}}$, where
$a \in (0,\rho)$, $x < \log(\rho - a)$, and $h>0$, and let the
polynomial $p_{\pm,n}(t) := (1 \pm t)(1-t^2)^n$. Then $p_{\pm,n}(t)$
is non-negative for all $t \in [-1,1]$ and all $n \geq 0$.
Applying~(\ref{Etrick}) gives that
\begin{equation}\label{Ediff}
\left| \sum_{k=0}^n \binom{n}{k} (-1)^k F(a + e^{x - 2kh}) \right|
\geq \left| \sum_{k=0}^n \binom{n}{k} (-1)^{n-k} 
F(-a + e^{x - (2k+1)h}) \right|.
\end{equation}
Let $H_{\pm,a}(x) := F(\pm a + e^x)$ and suppose $F$ is
smooth; dividing~(\ref{Ediff}) by $h^n$ and taking $h \to 0^+$, we
see that
\[
|H_{+,a}^{(n)}(x)| \geq |H_{-,a}^{(n)}(x)|.
\]
Since $H_{+,a}$ is real analytic, we conclude that the Taylor series
for~$H_{-,a}$ has a positive radius of convergence everywhere, so
$H_{-,a}$ is real analytic on~$(-\infty, \log( \rho - a ))$. The
change of variable $x = \log( y + a )$ has a convergent power-series
expansion for~$|y| < a$. It follows that $y \mapsto F( y )$ is real
analytic on~$(-\rho,\rho)$, hence is the restriction
of~$\widetilde{F}$.

When $F$ is not necessarily smooth, we may use a mollifier
argument. Fix $0 < \rho' < \rho$ and let
$G := F|_{(-\rho', \rho')}$. For any $\delta \in (0, \rho - \rho')$,
choose $g_\delta \in C^\infty(\R)$ such that $g_\delta$ is non-negative,
supported on $(0,\delta)$, and integrates to~$1$, and let
\[
F_\delta(x) := \int_0^\delta G(x + t) g_\delta(t) \std t
\qquad \text{for all } x \in (-\rho', \rho').
\]
As the function $x \mapsto G( t + x )$ satisfies hypothesis~(1) of the
theorem with $\rho$ replaced by~$\rho'$, so does the smooth function
$F_\delta$; let $\widetilde{F}_\delta$ be an analytic function on the
disc $D(0, \rho')$ which is absolutely monotonic on~$( 0, \rho' )$ and
agrees on $( -\rho', \rho' )$ with $F_\delta$. Since
\[
| F( x ) - F_\delta( x )| = 
\left|\int_0^\delta ( G(x) - G(x+t) ) g_\delta(t) \std t\right|
\leq \sup_{0 \leq t \leq \delta} |G(x) - G(x+t)|
\]
for all $x \in (-\rho', \rho')$, it follows that $F_\delta$ converges
to~$F$ locally uniformly on $(-\rho',\rho')$ as~$\delta \to 0^+$. The
function $\widetilde{F_\delta}$ is absolutely monotonic, so $|
\widetilde{F_\delta}( z ) | \leq \widetilde{F_\delta}( a )$ whenever
$| z | \leq a < \rho'$, and $\widetilde{F_\delta}( a ) \to F( a )$ as
$\delta \to 0^+$. Hence $\{ \widetilde{F_\delta}(z) : \delta > 0 \}$
is uniformly bounded on $\overline{D}(0,a)$, and therefore forms a
normal family. Thus for some sequence $\delta_n \to 0^+$, the
functions $\widetilde{F_{\delta_n}}$ converge locally uniformly to a
function $\widetilde{F}$ that is analytic on $D(0,\rho')$, and $F$
agrees with $\widetilde{F}$ on~$(-\rho',\rho')$. As this argument
holds for all $\rho' \in (0,\rho)$, the proof is complete.
\end{proof}

\begin{remark}\label{R2sided}
The proof of Theorem~\ref{T2sided-general} requires measures whose
support contains the point $1$, in order to be able to employ the
mollifier argument to move from continuous to smooth functions.
\end{remark}

\begin{remark}\label{Rembed}
Recall that Rudin \cite{Rudin59} showed that~$F$ must be analytic on
$(-1,1)$ and absolutely monotonic on $(0,1)$ if~$F[-]$
preserves positivity for the two-parameter family of Toeplitz matrices
defined in~(\ref{Etoeplitz}). A natural strategy to prove
Theorem~\ref{T2sided-general} would be to show that there exists
$\theta \in \R$ with $\theta / \pi$ irrational, such that the matrices
$(\cos ((i-j) \theta))_{i,j=1}^n$ can be embedded into positive Hankel
matrices, for all sufficiently large~$n$. However, this is not
possible: given $0 < m_1 < m_2$ such that $\cos (m_1 \theta) < 0$ and
$\cos (m_2 \theta) < 0$, if there were a measure
$\mu \in \meas^+([-1,1])$ such that
$\cos( m_j \theta ) = s_{k_j}(\mu)$ for $j=1$ and $j=2$, then, by the
Toeplitz property, $k_1$, $k_2$, and $k_1 + k_2$ must all be odd,
which is impossible.
\end{remark}

\section{Moment transformers on $[-1,0]$}\label{Smomentsm10}

We now study the structure of endomorphisms of $\moment([-1,0])$. The
following result strengthens Theorem~\ref{Tminus} and reveals that
such functions may be discontinuous at the origin, in contrast to
Theorem~\ref{T2sided-general}.

\begin{theorem}\label{Tminus-odd}
Given $u_0 \in (0,1)$ and $F : ( -\rho, \rho ) \to \R$, where
$0 < \rho \leq \infty$, the following are equivalent.
\begin{enumerate}
\item $F[-]$ maps $\momentr( \{ -1, -u_0 \})$ into
  $\moment((-\infty,0])$ and
\[
\momentr_4(\{ -1, 0 \}) \cup
\bigcup_{u \in (0,1)} \momentr_4(\{ -u \})
\]
into $\moment_4((-\infty,0])$.

\item $F[-]$ maps $\momentr([-1,0])$ into $\moment([-1,0])$.

\item There exists an absolutely monotonic entire function
$\widetilde{F}$ such that
\[
F(x) = \begin{cases}
\widetilde{F}(x) & \text{if } x \in (0,\rho),\\
0 & \text{if } x=0,\\
-\widetilde{F}(-x) \qquad & \text{if } x \in (-\rho,0).
\end{cases}
\]
In particular, the function $F$ is odd, but may be discontinuous
at~$0$.
\end{enumerate}
\end{theorem}

\begin{proof}

To show that $(3) \implies (2)$, note first that if
$\mu \in \meas^+([-1,0])$, so that $\mu = a \delta_0$ for some $a$,
then $F[H_\mu] = H_{F(a)\delta_0}$, so we may assume $\mu$ is not of
this form, whence the Hankel matrix $H_\mu$ has no zero entries, and
the moment sequence alternates in sign and is uniformly bounded, by
Theorem~\ref{Thamburger}. In particular,
\[
F(s_{2k}(\mu)) = \widetilde{F}(s_{2k}(\mu)) \quad \text{and} \quad
F(s_{2k+1}(\mu)) = -\widetilde{F}(-s_{2k+1}(\mu)) \qquad
( k \geq 0 ).
\]
Recalling the form of the Hankel matrix $H_{\delta_{-1}}$, it follows
that
\begin{equation}\label{Ecomposite}
F[H_\mu] = H_{\delta_{-1}} \circ 
\widetilde{F}[ H_{\delta_{-1}} \circ H_\mu]
\end{equation}
where $\circ$ denotes the entrywise matrix product. This shows (2)
because $F[-]$ is the composite of two operations: the map
$\widetilde{F}[-]$, which sends $\momentr([0,1])$ into
$\moment([0,1])$, by Theorem~\ref{T1sided-general}, and entrywise
multiplication by the matrix~$H_{\delta_{-1}}$, which maps $H_\mu$ for
some measure $\mu$ to the Hankel matrix of the reflection of~$\mu$
about the origin.

That $(2) \implies (1)$ is immediate. We now prove
$(1) \implies (3)$. Suppose (1) holds. Since
\[
F[H_{a \delta_0}] = (F(a) - F(0)) H_{\delta_0} + F(0) H_{\delta_1} =
H_{(F(a) - F(0))\delta_0 + F(0)\delta_1},
\]
the uniqueness in Theorem~\ref{Thamburger} gives that~$F(0) = 0$.

By considering only even rows and columns of Hankel matrices
corresponding to moments in $\momentr_4(\{ -u \})$,
$\momentr_4(\{ -1, 0 \})$, and
$\momentr(\{ -1, -u_0 \})$, we have embeddings
\begin{align*}
 & \momentr_2(\{ u^2 \}) \hookrightarrow \momentr_4(\{ -u \}), \\
 & \momentr_2(\{ 0, 1 \}) \hookrightarrow
\momentr_4(\{ -1, 0 \}), \\
\text{and} \quad &
\momentr(\{ 1, u_0^2 \}) \hookrightarrow \momentr(\{ -1, -u_0 \}).
\end{align*}
Thus $F[-]$ maps $\momentr_2(\{ u^2 \})$ into $\moment_2(\R)$,
$\momentr_2(\{ 0, 1 \})$ into $\moment_2(\R)$, and
$\momentr(\{ 1, u_0^2 \})$ into $\moment(\R)$.
Theorem~\ref{Treformulation} now gives that $F$ agrees with an
absolutely monotonic entire function $\widetilde{F}$ on~$( 0, \rho )$.

Next, considering $\momentr_2(\{ -1 \})$ gives that
$|F(-a)| \leq F(a)$ for any $a \in (0,\rho)$, whence $F$ is
locally bounded. In particular, $F$ maps $\momentr(\{-1\})$ into
$\moment([-1,0])$, by Theorem~\ref{Thamburger}.

We conclude by showing that $F$ is odd. Let $\mu = a \delta_{-1}$ for
some~$a \in (0, \rho)$ and note that $p_n(t) = (-t)^n (1+t)$ is
non-negative on~$[-1,0]$ for any non-negative integer~$n$. If
$F[\bs(\mu)] = \bs(\sigma)$, then, by applying~(\ref{Etrick}),
\begin{align*}
0 \leq \int_{-1}^0 p_n(t) \std\sigma & = 
(-1)^n (F(s_n(a \delta_{-1})) + F(s_{n+1}(a \delta_{-1})) \\
 & = (-1)^n (F((-1)^n a) + F((-1)^{n+1} a)).
\end{align*}
Taking $n = 0$ and~$1$ gives that
$0 \leq F(a) + F(-a) \leq 0$, and the final claim follows.
\end{proof}

Theorem \ref{Tminus-odd} has the following consequence.

\begin{corollary}\label{Cnew}
Define a \emph{checkerboard} matrix to be any real matrix
$A = (a_{ij})$ such that $(-1)^{i+j} a_{ij} > 0$ for all $i$, $j$. Given
a function $F : \R \to \R$, the following are equivalent. 
\begin{enumerate}
\item Applied entrywise, $F$ maps the set of positive Hankel checkerboard
matrices of all dimensions into itself.
\item Applied entrywise, $F$ maps the set of positive checkerboard
matrices of all dimensions into itself.
\item $F$ is odd and agrees on $(0,\infty)$ with an absolutely
monotonic entire function.
\end{enumerate}
\end{corollary}

We conclude this section with an even analogue of
Theorem~\ref{Tminus-odd}.

\begin{theorem}\label{Tminus-even}
Given $u_0 \in (0,1)$ and $F : ( -\rho, \rho ) \to \R$, where
$0 < \rho \leq \infty$, the following are equivalent.
\begin{enumerate}
\item $F[-]$ maps $\momentr( \{ -1, -u_0 \})$ into
$\moment([0,\infty))$ and
\[
\momentr_4(\{ -1, 0 \}) \cup \bigcup_{u \in (0,1)}
\momentr_4(\{ -u \})
\]
into $\moment_4([0,\infty))$.

\item $F[-]$ sends $\momentr([-1,0])$ to $\moment([0,1])$.

\item There exists an absolutely monotonic entire function
$\widetilde{F}$ such that
\[
F(x) = \begin{cases}
\widetilde{F}(x) & \text{if } x \in (0,\rho),\\
\widetilde{F}(-x) \qquad & \text{if } x \in (-\rho,0).
\end{cases}
\]
Moreover, $0 \leq F(0) \leq \lim_{\epsilon \to 0} F(\epsilon)$.
\end{enumerate}
\end{theorem}

\begin{proof}
This is similar to the proof of Theorem \ref{Tminus-odd}; to show that
$(1) \implies (3)$, one may use the polynomials $p_n(t) =
t^n(1-t)$. We omit further details.
\end{proof}

\section{Transformers with compact domain}\label{Sdomain}

The goal of this section is to show how results in the previous
sections can be refined when the moments are contained in a compact
domain. Indeed, when the domain of $F$ is a compact interval~$I$, the
situation is more complex; absolute monotonicity, or even continuity
of $F$, does not extend automatically from the interior of $I$ to its
end~points. This was already observed by Rudin via specific
counterexamples; see Remark~(a) at the end of \cite{Rudin59}. To the
best of our knowledge, characterization results in this setting are
not known.

We now take a closer look at this phenomenon. We begin by
characterizing the functions preserving positivity of Hankel matrices
in $\bp_N(I)$ for all $N$, where $I = [0,\rho]$ and
$0 < \rho < \infty$.

\begin{proposition}\label{P1sided}
Suppose $F : I \to \R$, where $I = [0,\rho]$ and
$0 < \rho < \infty$. The following are equivalent.
\begin{enumerate}
\item $F[-]$ preserves positivity on all positive Hankel matrices with
entries in~$I$.
\item $F$ is absolutely monotonic on $[0,\rho)$ and
$F(\rho) \geq \lim_{x \to \rho^-} F(x)$.
\item $F[-]$ preserves positivity on all positive matrices with
entries in~$I$.
\end{enumerate}
If, instead, $I = [0,\rho)$ where $0 < \rho \leq \infty$, then the
same equivalences hold, with (2) replaced by the requirement that~$F$
is absolutely monotonic on~$[0,\rho)$.
\end{proposition}

Note the contrast with Theorem~\ref{T1sided-general}: if $F[-]$ is
required only to preserve positive Hankel matrices arising from moment
sequences, then $F$ may be discontinuous at $0$, but this cannot occur
here.

\begin{proof}
Clearly $(3) \implies (1)$. Next, suppose $(1)$ holds and note that
$F$ is absolutely monotonic on $(0,\rho)$, by
Theorem~\ref{T1sided-general}. Consider the positive Hankel matrices
\[
H_a := \begin{pmatrix} a & 0 & a\\ 0 & a & a\\ a & a & 2a
\end{pmatrix}, \qquad \text{where } a \in [0, \rho/2).
\]
As $F[H_a]$ is positive, so
$0 \leq F( 0 ) \leq F^+(0) := \lim_{a \to 0^+} F(a)$. Furthermore,
\[
0 \leq \lim_{a \to 0^+} \det F[H_a] = -F^+(0) (F(0) - F^+(0))^2,
\]
whence $F(0) = F^+(0)$, and $F$ is right continuous at the
origin. Finally, considering the first two leading principal minors of
the Hankel matrix for the measure $(\rho-a) \delta_1 + a \delta_0$,
where $a \to \rho^-$, gives that
$F(\rho) \geq \lim_{a \to \rho^-} F(a)$. Hence $(1) \implies (2)$.

Finally, suppose $(2)$ holds. We first claim that if
$A \in \bp_N((-\infty,\rho])$ then the entries of $A$ equalling $\rho$
form a (possibly empty) block diagonal submatrix, upon suitably
relabelling the indices. Indeed,
\begin{equation}\label{Eudcc}
0 \leq \det \begin{pmatrix} \rho & \rho & a \\ \rho & \rho & \rho \\ a &
\rho & \rho \end{pmatrix} = - \rho (\rho - a)^2,
\quad \text{so} \quad a = \rho.
\end{equation}
Now let $B_A$ be the block-diagonal matrix with $(i,j)$th
entry equal to~$1$ if $a_{i j}=\rho$ and~$0$ otherwise. If $g$ is the
continuous extension of~$F|_{[0,\rho)}$ to $\rho$, then
\[
F[A] = g[A] + (F(\rho) - g(\rho)) B_A \geq 0,
\]
since both matrices are positive semidefinite. Hence
$(2) \implies (3)$.

Finally, when $I = [0,\rho)$, that $(2) \implies (3) \implies (1)$ is
immediate, and $(1) \implies (2)$ is shown as above.
\end{proof}

\begin{remark}\label{R1sided}
A similar argument to Proposition~\ref{P1sided} reveals that $F[-]$
preserves positivity on the set
$\{ \bs(\mu) \in \moment([0,1]) : s_0(\mu) \in [0,\rho] \}$ if and
only if $F$ is absolutely monotonic on~$(0,\rho)$ and such that
$0 \leq F(0) \leq \lim_{\epsilon \to 0^+} F(\epsilon)$ and
$\lim_{x \to \rho^-} F(x) \leq F(\rho)$.
\end{remark}

We next examine the case where the domain of $F$ is a symmetric
compact interval $[-\rho,\rho]$. The functions preserving positivity
of Hankel matrices when applied entrywise are completely characterized
as follows.

\begin{proposition}\label{PclosedHankel}
Suppose $F : I \to \R$, where $I = [-\rho,\rho]$ and
$0 < \rho \leq \infty$. The following are equivalent.
\begin{enumerate}
\item $F[-]$ preserves positivity on all positive Hankel matrices with
entries in~$I$.
\item $F[-]$ preserves positivity on all positive Hankel matrices with
entries in~$I$ that arise from moment sequences.
\item $F$ is real analytic on $(-\rho,\rho)$, absolutely monotonic
on~$(0,\rho)$, and such that
\[
F(\rho) \geq \lim_{x \to \rho^-} F(x) \quad \text{and} \quad
|F(-\rho)| \leq F(\rho).
\]
\end{enumerate}
\end{proposition}

\begin{proof}
That $(1) \implies (2)$ is immediate, while $(2) \implies (3)$ follows
from the extension of Theorem~\ref{T2sided} given by
Theorem~\ref{T2sided-general}, and the proofs of
Proposition~\ref{P1sided} and Lemma~\ref{Lbounded}. Finally, if (3)
holds, then (1) follows by Proposition~\ref{P1sided}, the Schur
product theorem, and the following claim.

\emph{The only Hankel matrix in $\bp_{N+1}([-\rho,\rho])$ with an entry
${-\rho}$ is the checkerboard matrix with $(i,j)$th entry
$(-1)^{i+j} \rho$.}

To prove the claim, let the rows and columns of the positive Hankel
matrix $A$ be labelled by $0$, \dots, $N$, and suppose
$a_{ij} = -\rho$. Then $i+j$ is odd and $a_{ll} = a_{l+1,l+1} = \rho$,
where  $2l+1 = i+j$. Repeatedly considering principal $2 \times 2$
minors shows that $a_{pq} = \rho$ if $p+q$ is even. Now let
$m, n \in [0,N]$ be odd, with $m < n$, and  denote by~$C$ the
principal $3 \times 3$ minor of $A$ corresponding to the labels $0$,
$m$, and~$n$. Writing
\[
C = \begin{pmatrix} \rho & a_{0m} & \rho\\ a_{0m} & \rho & a_{0n}\\ \rho
& a_{0n} & \rho \end{pmatrix},
\]
we have that $0 \leq \det C = -\rho(a_{0m} - a_{0n})^2$,
whence $a_{0m} = a_{0n}$. Taking $m$ or $n$ to equal $i+j$ shows that
these entries equal $-\rho$, which gives the claim.
\end{proof}

We end this section by considering functions preserving positivity for
all matrices in $\bigcup_{N \geq 1} \bp_N([-\rho,\rho])$.
Theorem~\ref{T2sided-general} implies that every such function
$F$ is real analytic when restricted to $(-\rho,\rho)$, and absolutely
monotonic on~$(0,\rho)$. The following result provides a sufficient
condition for $F$ to preserve positivity, which is also necessary if the
analytic restriction is odd or even.

\begin{proposition}\label{Pboundary}
Given $\rho \in (0,\infty)$, let $I = [-\rho,\rho]$ and suppose
$F : I \to \R$ is real analytic on $(-\rho,\rho)$, absolutely
monotonic on~$(0,\rho)$, and such that the limits
\smash[b]{$\displaystyle \lim_{x \to \rho^-} F(\pm x)$} both exist and
are finite. If
\begin{equation}\label{Ejump}
\left| F(-\rho) - \lim_{x \to -\rho^+} F(x) \right| \leq
F(\rho) - \lim_{x \to \rho^-} F(x),
\end{equation}
then $F[-]$ preserves positivity on the space of positive matrices
with entries in $I$. The converse holds if $F|_{(-\rho,\rho)}$ is
either odd or even.
\end{proposition}

The inequality~(\ref{Ejump}) says that any jump in~$F$ at~$-\rho$ is
bounded above by the jump at~$\rho$, which is non-negative.

\begin{proof}
Let $g$ denote the continuous function on $[-\rho,\rho]$ which agrees
with $F$ on $(-\rho,\rho)$, and let the jumps
$\Delta_\pm := F(\pm \rho) - g(\pm \rho)$. Then~(\ref{Ejump}) is
equivalent to $|\Delta_-| \leq \Delta_+$.

By the Schur product theorem and Proposition \ref{P1sided}, $F[-]$
preserves positivity on $\bp_N((-\rho,\rho])$ for all $N$. Now suppose
$A \in \bp_N([-\rho,\rho])$ has some entry equal to $-\rho$, where
$N\geq1$. Then the entries of $A$ with modulus~$\rho$ form a block
diagonal submatrix upon suitable relabelling of indices. This follows
from the argument given in the proof of Proposition~\ref{P1sided},
applied to the $\rho^2$-entries of~$A \circ A$.
Given this, and after further relabelling of indices, each block
submatrix is of the form
\[
\begin{pmatrix}
\rho \bone_{n_j \times n_j} & -\rho \bone_{n_j \times m_j} \\
-\rho \bone_{m_j \times n_j} & \rho \bone_{m_j \times m_j}
\end{pmatrix},
\]
by the main result in \cite{Matrix01psd}, where $j=1$, \dots, $r$.
Then
\[
F[A] = g[A] + B', \qquad \text{where } B' = \oplus_{j=1}^k 
\begin{pmatrix} \Delta_+ \cdot \bone_{n_j \times n_j} &
\Delta_- \cdot \bone_{n_j \times m_j} \\
\Delta_- \cdot \bone_{m_j \times n_j} &
\Delta_+ \cdot \bone_{m_j \times m_j}
\end{pmatrix},
\]
and this is positive semidefinite, by~(\ref{Ejump}). Thus $F[-]$
preserves $\bigcup_{N \geq 1} \bp_N([-\rho,\rho])$.

For the converse, we show that (\ref{Ejump}) holds if~$F[-]$ preserves
positivity on just the set~$\bp_3([-\rho,\rho])$
and~$F|_{(-\rho,\rho)}$ is odd or even. Note first that $\Delta_+ \geq
0$, working with $2 \times 2$ matrices as above. Next, consider the
positive matrix
\[
A := \begin{pmatrix}
a^2 / \rho & -a & a \\ -a & \rho & -\rho \\ a & -\rho & \rho
\end{pmatrix},
\]
and note that
\[
0 \leq \lim_{a \to \rho^-} \det F[A] = \begin{vmatrix}
g(\rho) & g(-\rho) & g(\rho) \\
g(-\rho) & F(\rho) & F(-\rho) \\
g(\rho) & F(-\rho) & F(\rho)
\end{vmatrix} = \Delta_+(g(\rho) F(\rho) - g(-\rho)^2) -
g(\rho) \Delta_-^2.
\]
It follows that $\Delta_-^2 g(\rho) \leq \Delta_+^2 g(\rho)$ if
$g(\rho^2) = g(-\rho)^2$, so if $g = F|_{(-\rho,\rho)}$ is odd or
even. This gives the result.
\end{proof}

\begin{remark}
Propositions~\ref{PclosedHankel} and~\ref{Pboundary} indicate the
existence of functions discontinuous at~$\pm \rho$ which preserve
positivity for Hankel matrices, but not all matrices, in contrast to
the one-sided setting of Proposition~\ref{P1sided}.

Indeed, if $g$ is an odd or even function which is continuous
on~$[-\rho,\rho]$ and absolutely monotonic on~$(0,\rho)$,
define $F$ to be equal to $g$ on $(-\rho, \rho]$, and take $F(-\rho)$
to be any element of $(-F(\rho), F(\rho)]$. Then $F$ preserves
positivity on all Hankel matrices with entries in $[-\rho, \rho]$, but
does not preserve positivity
on~$\bigcup_{N \geq 1} \bp_N([-\rho,\rho])$.
\end{remark}

\section{Multivariable generalizations}\label{Smulti}

In this section we classify the preservers of moments arising from
admissible measures in higher-dimensional Euclidean space, both in
their totality and by considering their marginals.

\subsection{Transformers of multivariable moment sequences}

The initial generalization to higher dimensions of our
characterization of moment-preserving functions raises no
complications. However, the failure of Hamburger's theorem in higher
dimensions, that is, the lack of a characterization of moment
sequences by positivity of an associated Hankel-type kernel, means
some extra work is required. Below, we isolate this additional
challenge and provide a generalization of our main result.

Let $\mu$ be a non-negative measure on $\R^d$, where $d \geq 1$, which
has moments of all orders; as before, such measures will be termed
\emph{admissible}. The multi-index notation
\[
\bx^\alpha = x_1^{\alpha_1} \dots x_d^{\alpha_d}  \qquad (\bx \in \R^d)
\]
allows us to define the moment family
\[
s_\alpha(\mu) = \int \bx^\alpha \std\mu(\bx) \qquad 
( \alpha \in \nnZ^d ),
\]
where $\nnZ$ denotes the set $\{ 0, 1, 2, \ldots \}$ of
non-negative integers. As before, we focus on measures with uniformly
bounded moments, so that
\[
\sup_{\alpha \in \nnZ^d} |s_\alpha(\mu)| < \infty, 
\]
or, equivalently, ${\rm supp}(\mu) \subset [-1,1]^d$. In line with above,
we let $\moment(K)$ denote the set of all moment families of
admissible measures supported on $K \subset \R^d$.

\begin{theorem}\label{Teuclidean}
The map $F[-] : \R \to \R$ maps $\moment([-1,1]^d)$ into itself
if and only if~$F$ is absolutely monotonic and entire.
\end{theorem}

\begin{proof}
Any admissible measure $\mu$ on $[-1,1]$ pushes forward to an
admissible measure~$\widetilde{\mu}$ on~$[-1,1]^d$ via the canonical
embedding onto the first coordinate. If $F$ maps $\moment([-1,1]^d)$
to itself then there exists an admissible measure $\widetilde{\sigma}$
on~$[-1,1]^d$ such that
$F(s_\alpha(\widetilde{\mu}))=s_\alpha(\widetilde{\sigma})$ for all
$\alpha\in\nnZ^d$, and a short calculation shows that
$F[s_n(\mu)]=s_n(\sigma)$ for all $n \in \nnZ$, where $\sigma$ is the
pushforward of $\widetilde{\sigma}$ under the projection onto the
first coordinate. Theorem~\ref{T2sided-general} now gives that $F$ is
as claimed.

To prove the converse, we need to explore the structure of the set
$\moment([-1,1]^d)$. Denote the generators of the semigroup $\nnZ^d$
by setting
\[
\bone_j := (0,\dots,0,1,0,\dots, 0),
\]
with $1$ in the $j$th position. A multisequence of real numbers
$(s_\alpha)_{\alpha \in \nnZ^d}$ is the moment sequence of an
admissible measure supported on~$[-1,1]^d$ if and only if the weighted
Hankel-type kernels
\[
(s_{\alpha+\beta}), \ 
(s_{\alpha+\beta}-s_{\alpha+\beta+2\bone_j}), \quad 1 \leq j \leq d,
\]
indexed over $\alpha$, $\beta \in \nnZ^d$ are positive
semidefinite \cite{Putinar}.

Now suppose $F$ is absolutely monotonic and entire; given a
multisequence $s_\alpha$ subject to these positivity constraints, we
have to check that the multisequence $F(s_\alpha)$ satisfies the same
conditions.

As $F$ is absolutely monotonic, Schoenberg's Theorem~\ref{TSchoenberg}
gives that the kernels $(\alpha,\beta) \mapsto F(s_{\alpha+\beta})$
and $(\alpha,\beta) \mapsto F(s_{\alpha+\beta+2\bone_j})$ are positive
semidefinite. It remains to prove that the kernel
\[
(\alpha,\beta) \mapsto
F(s_{\alpha+\beta}) - F(s_{\alpha+\beta+2\bone_j})
\]
is positive semidefinite, for $1 \leq j \leq d$. However, as $F$ has
the Taylor expansion $F(x) = \sum_{n=0}^\infty c_n x^n$, with
$c_n \geq 0$ for all~$n \in \nnZ$, it is sufficient to check that the
kernel
\[
(\alpha,\beta) \mapsto
(s_{\alpha+\beta})^{\circ n} - (s_{\alpha+\beta+2\bone_j})^{\circ n}
\]
is positive semidefinite for any $n \in \nnZ$. This follows from a
repeated application of the Schur product theorem: if matrices $A$ and
$B$ are such that $A \geq B \geq 0$, then
\[
A^{\circ n} \geq A^{\circ (n-1)} \circ B \geq A^{\circ (n-2)} \circ
B^{\circ 2} \geq \cdots \geq B^{\circ n}.
\qedhere
\]
\end{proof}

This proof also shows that the transformers of $\moment([-1,1]^d)$
into $\moment(\R^d)$ are the same absolutely monotonic entire
functions. On the other hand, we will see in Section~\ref{SLaplace}
that, in general, a mapping~$F$ as in Theorem~\ref{Teuclidean} does
not preserve the semi-algebraic supports of the underlying measures.

\subsection{Transformers of moment-sequence tuples: the positive
orthant case}\label{SfaceAM}

Our next objective is to characterize functions $F: \R^m \to \R$ which
map tuples of moments $(s_k(\mu_1), \dots, s_k(\mu_m))$ arising from
admissible measures on~$\R$, to a moment sequence $s_k(\sigma)$ for
some admissible measure~$\sigma$ on $\R$. This is a multivariable
generalization of Schoenberg's theorem which we will demonstrate under
significantly relaxed hypotheses.

More precisely, we will study the preservers $F : I^m \to \R$,
where $m \geq 1$ is a fixed integer, and
\begin{equation}\label{Edomain}
I = (0, \rho) \text{ or } [0,\rho) \text{ or } (-\rho, \rho), \qquad
\text{where } 0 < \rho \leq \infty.
\end{equation}
Note that $F : I^m \to \R$ acts entrywise on any $m$-tuple of
$N \times N$ matrices with entries in $I$, so that
\begin{equation}
F[-] : \bp_N(I)^m \to \R^{N \times N}, \qquad
F[ A_1, \ldots, A_m ]_{ij} :=
F( a_{1, ij}, \ldots, a_{m, ij} ).
\end{equation}
By the Schur product theorem, every real analytic function $F$ of the
form
\begin{equation}
F( \bx ) = \sum_{\alpha \in \nnZ^m} c_\alpha \bx^\alpha \qquad
( \bx \in I^m )
\end{equation}
preserves positivity on $\bp_N( I )^m$ if $c_\alpha \geq 0$ for all
$\alpha \in \nnZ^m$. The reverse implication was shown by FitzGerald,
Micchelli, and Pinkus in~\cite{fitzgerald} for $\rho = \infty$, and
can be thought of as a multivariable version of Schoenberg's theorem.
We now explain how results on several real and complex variables can
be used to generalize the work in previous sections to this
multivariable setting, including over bounded domains in the original
spirit of Schoenberg and Rudin. Namely, we characterize functions
mapping tuples of positive Hankel matrices into themselves. Of course,
this is equivalent to mapping tuples of moment sequences of admissible
measures into the same set.

First we need some notation and terminology. Given $I$ as in
(\ref{Edomain}), suppose the sets $K_1$, \ldots, $K_m \subset \R$ are
such that all sequences in $\momentr( K_j )$ have entries in $I$, for
$j = 1$, \ldots, $m$. A function $F : I^m \to \R$ acts on $m$-tuples
of moment sequences of measures in
$\momentr(K_1) \times \cdots \times \momentr(K_m)$ to produce real
sequences, so that
\begin{equation}
F[\bs(\mu_1), \ldots, \bs(\mu_m)]_k :=
F( s_k(\mu_1), \ldots, s_k(\mu_m) ) \qquad ( k \in \nnZ ).
\end{equation}

Given $I' \subset \R^m$, a function $F : I' \to \R$ is
\emph{absolutely monotonic} if $F$ is continuous on~$I'$, and for any
interior point $\bx \in I'$ and $\alpha \in \nnZ^m$, the mixed partial
derivative $D^\alpha F(\bx)$ exists and is non-negative. As usual, for
a tuple $\alpha = (\alpha_1, \ldots, \alpha_m) \in \nnZ^m$, we set
\[
D^\alpha F(\bx) := \frac{\partial^{|\alpha|}}{\partial x_1^{\alpha_1}
\cdots \partial x_m^{\alpha_m}} \; F(x_1, \dots, x_m),
\qquad \text{where } |\alpha| := \alpha_1 + \dots + \alpha_m.
\]
The analogue of Bernstein's Theorem for the multivariable case is proved
and put in its proper context in Bochner's book; see
\cite[Theorem~4.2.2]{Bochner-book}.

Our first observation is the connection between functions acting on
tuples of moment sequences and on the corresponding Hankel matrices.
Given admissible measures $\mu_1$, \dots, $\mu_m$ and $\sigma$
supported on the real line, it is clear that
\[
F[\bs(\mu_1), \dots, \bs(\mu_m)] = \bs(\sigma) \quad \iff \quad
F[H_{\mu_1}, \dots, H_{\mu_m}] = H_\sigma.
\]
In particular, equality holds at each finite truncation, that is, for
the corresponding leading principal $N \times N$ submatrices, for any
$N \geq 1$. We will henceforth switch between moment sequences and
positive Hankel matrices without further comment.

We begin by considering the case of matrices with positive entries,
arising from tuples of sequences in $\momentr([0,1])$. To state and
prove the main result in this subsection, we require a preliminary
technical result.

\begin{lemma}\label{Lgeneric}
Given an integer $m \geq 1$, let $\cY_m$ denote the set
of all $\by = (y_1, \dots, y_m)^T \in (0,1)^m$ such that the scalars
\[
\by^\alpha := \prod_{l=1}^m y_l^{\alpha_l}
\]
are distinct for all $\alpha \in \nnZ^m$. Then the complement
of~$\cY_m$ in $(0,1)^m$ has zero $m$-dimensional Lebesgue measure.
\end{lemma}

\begin{proof}
Let
\[
\cX := \{ \bx = \log \by \in ( -\infty, 0 )^m : \bx \not\perp \alpha
\text{ for all } \alpha \in \Z^m \setminus \{ ( 0, \ldots, 0 ) \} \}.
\]
The complement of $\cX$ in $(-\infty,0)^m$ is a countable union of
hyperplanes, and so has measure zero. The result now follows since
$\cY_m$ is the image of~$\cX$ under a smooth map.
\end{proof}

The new notion of a \emph{facewise absolutely monotonic function}
on~$[0,\rho)^m$ plays an important role in our next result. In
order to define it, recall that the truncated orthant~$[0,\rho)^m$ is
the truncation of a convex polyhedron, and as such, is the disjoint
union of the relative interiors of its faces. These faces are in
bijection with subsets of $[m] := \{ 1, \ldots, m \}$ via the
mapping
\begin{equation}\label{Epolyhedral}
J \mapsto [0,\rho)^J := \{ (x_1, \ldots, x_m) \in [0,\rho)^m :
x_l = 0 \text{ for all } l \not\in J \},
\end{equation}
and this face has relative interior
$(0,\rho)^J \times \{ 0 \}^{[m] \setminus J}$.

\begin{definition}
A function $F : [ 0, \rho )^m \to \R$, where $m \geq 1$ and
$0 < \rho \leq \infty$, is \emph{facewise absolutely monotonic} if,
for each set of indices $J \subset [m]$, the function~$F$ agrees
on~$(0,\rho)^J \times \{ 0 \}^{[m] \setminus J}$ with an absolutely
monotonic function $g_J$ on $(0,\rho)^J$. Here and henceforth, we
identify without further comment $(0,\rho)^J$ and
$(0,\rho)^J \times \{ 0 \}^{[m] \setminus J}$.
\end{definition}

In other words, a facewise absolutely monotonic function is piecewise
absolutely monotonic, with the pieces being the relative interiors of
the faces of the truncated polyhedral cone $[0,\rho)^m$. The
following example illustrates this in the case $m = 2$.

\begin{example}\label{Eexample}
Let
\[
F(x_1, x_2) := \begin{cases}
x_1^2 + x_2^2 + 1 \qquad & \text{if } x_1, x_2 > 0,\\
2 x_1 \qquad & \text{if } x_1 > 0, x_2 = 0,\\
x_2^2 + 1 \qquad & \text{if } x_1 = 0, x_2 > 0,\\
0 \qquad & \text{if } x_1 = x_2 = 0.
\end{cases}
\]
Then $F$ is facewise absolutely monotonic, with
\[
g_\emptyset = 0, \qquad g_{\{1\}}(x_1) = 2 x_1, \qquad
g_{\{2\}}(x_2) = x_2^2 + 1, \quad \text{and} \quad
g_{\{ 1, 2 \}}(x_1,x_2) = x_1^2 + x_2^2 + 1.
\]
In this example, and, in fact, for every facewise absolutely monotonic
function, the function $g_J$ extends to an absolutely monotonic
function on the closure~$[0,\rho)^J$ of its domain, for all
$J \subset [m]$. We denote this extension by~$\widetilde{g}_J$.

Furthermore, for Example~\ref{Eexample}, the functions
$\widetilde{g}_J$ satisfy a form of monotonicity that is compatible
with the partial order on their labels:
\begin{equation}\label{Eextension}
K \subset J \subset [m] \quad \implies \quad 
0 \leq \widetilde{g}_K \leq \ \widetilde{g}_J \quad 
\text{on } [0,\rho)^K.
\end{equation}
A word of caution: while $\widetilde{g}_{\{1\}}(x_1) \leq
\widetilde{g}_{\{ 1, 2 \}}(x_1,0)$ for all $x_1 \geq 0$, it is
not true that the difference of these functions is absolutely
monotonic on~$[0,\rho)$.
\end{example}

With this definition and example in hand, together with
Lemma~\ref{Lgeneric}, we can now characterize the preservers of
tuples of moment sequences in~$\momentr([0,1])$.

\begin{theorem}\label{T1sided-multi}
Let $F : [0,\rho)^m \to \R$, where $m \geq 1$ and
$0 < \rho \leq \infty$, and fix
$\by = ( y_1, \ldots, y_m )^T \in \cY_m$, as in
Lemma~\ref{Lgeneric}. The following are equivalent.
\begin{enumerate}
\item $F[-]$ maps
$\momentr(\{ 1, y_1 \}) \times \cdots \times \momentr(\{ 1, y_m \})
\cup \momentr(\{ 0, 1 \})^m$ into $\moment(\R)$, and 
\[
F((a_1, \dots, a_m)) F((b_1, \dots, b_m)) \geq F \left( ( \sqrt{a_1 b_1},
\dots, \sqrt{a_m b_m} ) \right)^2
\]
for all $a_1$, \ldots, $a_m$, $b_1$, \ldots, $b_m \in [0,\rho)$.

\item $F[-]$ maps $\momentr([0,1])^m$ into $\moment([0,1])$.

\item $F$ is facewise absolutely monotonic, and the functions
$\{ g_J : J \subset [m] \}$ satisfy the monotonicity
condition~(\ref{Eextension}).
\end{enumerate}
\end{theorem}

Reformulating this result, as in the one-dimensional case above, it
suffices to work only with Hankel matrices of rank at most two.
Moreover, Theorem~\ref{T1sided-general} is precisely
Theorem~\ref{T1sided-multi} when $m = 1$. The proof
builds on Theorem~\ref{T1sided-general}; however, the higher
dimensionality introduces several additional challenges.

A large part of Theorem~\ref{T1sided-multi} can be deduced from the
following reformulation on the open cell in the positive orthant.

\begin{theorem}\label{Thorn-hankel2}
Fix $\rho \in ( 0, \infty ]$, an integer $m \geq 1$ and a point
$\by = (y_1, \dots, y_m)^T \in \cY_m$, as in Lemma~\ref{Lgeneric}. For
$1 \leq l \leq m$ and $N \geq 1$, let
\begin{align*}
\bu_{l,N} & := (1, y_l, \dots, y_l^{N-1})^T, \\
\text{and} \quad \mathsf{H}^+_l(N) & :=
\{ a \bone_{N \times N} + b \bu_{l,N} \bu_{l,N}^T : 
a \in (0,\rho), \ b \in [0,\rho-a) \}.
\end{align*}
If the function $F : (0,\rho)^m \to \R$ is such that $F[-]$ preserves
positivity on~$\bp_2((0,\rho))^m$ and on
$\mathsf{H}^+_1(N) \times \dots \times \mathsf{H}^+_m(N)$
for all~$N \geq 1$, then $F$ is absolutely monotonic and is the
restriction of an analytic function on $D(0,\rho)^m$.
\end{theorem}

\begin{remark}
As noted in Remark~\ref{Rhorn-hankel} for the one-variable
case, the proof of Theorem~\ref{Thorn-hankel2} goes through under a
weaker hypothesis, with the test sets replaced by
the set of rank-one $m$-tuples $\bp_2^1((0,\rho))^m$ and the set
\begin{equation}\label{Emulti0}
\left\{ \left( \begin{pmatrix} a_1 & b_1 \\ b_1 & b_1 \end{pmatrix},
\ldots, \begin{pmatrix} a_m & b_m \\ b_m & b_m \end{pmatrix}
\right) : 0 < b_l < a_l < \rho, \ 1 \leq l \leq m \right\}.
\end{equation}
The matrices in $\mathsf{H}^+_l( N )$ and (\ref{Emulti0}) are
precisely the truncated moment matrices of admissible measures
supported on~$\{ 1, y_l \}$ and on $\{ 0, 1 \}$, respectively.
\end{remark}

\begin{proof}[Proof of Theorem~\ref{Thorn-hankel2}]
We begin by recording a few basic properties of~$F$. First, either $F$
is identically zero, or it is everywhere positive on $(0,\rho)^m$;
this may be shown similarly to the proof of
Theorem~\ref{Thorn-hankel}.  Moreover, using only tuples from
$\bp_2^1((0,\rho))$ and~(\ref{Emulti0}), as well as the hypotheses,
one can argue as in the proof of Theorem~\ref{Thorn-hankel}, and show
that $F$ is continuous on~$(0,\rho)^m$.

Next, given
$\bc = (c_1, \dots, c_m)^T \in (0,\rho)^m$, the function $g$ such that
\[
g( \bx ) := F( \bx + \bc ) \qquad \text{for all }
\bx \in ( 0, \rho - c_1 ) \times \dots \times ( 0, \rho - c_m )
\]
satisfies the same hypotheses as $F$, but with
$\rho$ replaced by $\rho - c_l$ in each $\mathsf{H}^+_l(N)$, and
with $\bp_2^1((0,\rho))^m$ replaced by
$\bp_2^1((0,\rho - c_1)) \times \cdots \times \bp_2^1((0,\rho - c_m))$.
Therefore, as in the proof of \cite[Theorem~2.1]{fitzgerald}, a
mollifier argument reduces the problem to considering only smooth~$F$. We
now follow the proof of \cite[Proposition~2.5]{fitzgerald}, but with
suitable modifications imposed by the weaker hypotheses.

Given $r \geq 0$, we take $N \geq \binom{r+m}{m}$, and let
$\by'_l := (1, y_l, \dots, y_l^{N-1})^T$ for $1 \leq l \leq m$. Fix
some $\bc \in (0,\rho)^m$, choose $b_l \in (0, \rho - c_l)$ for
all~$l$ and let
\[
A_l := c_l \bone_{N \times N} + b_l \by'_l \by'_l{}^T \in
\mathsf{H}^+_l(N),
\]
so that $F[A_1, \dots, A_m] \in \bp_N(\R)$. We now use
Lemma~\ref{Lgeneric}: since $\by \in \cY_m$ and
$N \geq \binom{r+m}{m}$ by assumption, for each $\beta \in \nnZ^m$ with
$|\beta| \leq r$ we can choose $\bv_\beta \in \R^N$ such that
\[
\bv_\beta \perp
(1, \by^\alpha, \by^{2 \alpha}, \dots, \by^{(N-1) \alpha})^T \qquad
\text{for all } \alpha \in \nnZ^m \setminus \{ \beta \}
\text{ with } |\alpha| \leq r,
\]
and $(1, \by^\beta, \dots, \by^{(N-1) \beta}) \bv_\beta = 1$. An
application of Taylor's theorem (similar to its use in
Proposition~\ref{Phorn-hankel} or \cite[Proposition~2.5]{fitzgerald})
now gives that the derivative $D^\beta F(\bc) \geq 0$. Thus $F$ is
absolutely monotonic on~$(0,\rho)^m$, and Schoenberg's observation
\cite[Theorem~5.2]{Schoenberg33} implies that~$F$ is the restriction
to~$(0,\rho)^m$ of an analytic function on~$D(0,\rho)^m$.
\end{proof}

With this result in hand, we can now proceed.

\begin{proof}[Proof of Theorem \ref{T1sided-multi}]
Clearly, $(2) \implies (1)$.

We will show $(1) \implies (3)$ by induction on $m$. As
noted above, the case $m=1$ is precisely Theorem~\ref{T1sided-general}.
For the induction step, we first restrict~$F$ to the
relative interior of any face of the truncated polyhedron
$[0,\rho)^m$, say~$(0,\rho)^J$ for some $J \subset [m]$. The induction
hypothesis and Theorem~\ref{Thorn-hankel2} give that $F$ is facewise
absolutely monotonic, so $F \equiv g_J$ on~$(0,\rho)^J$,
with~$g_J$ absolutely monotonic. To see that (\ref{Eextension}) holds,
we claim that, for any pair of subsets
$L \subset K \subsetneq J \subset [m]$,
\[
\widetilde{g}_K(\bx) \leq \widetilde{g}_J(\bx) \qquad
\text{whenever } \bx \in (0,\rho)^L \subset [0,\rho)^m.
\]
For ease of exposition, we show this for an illustrative example; the
general case follows with minimal modification. Suppose
$J = \{ 1, 2, 3 \}, K = \{ 1, 2 \}$, and $L = \{ 1 \}$. For
any~$(x_1,0,0) \in (0,\rho)^L$, we set
\[
(a_1,a_2,a_3) := (x_1,x_2,x_3) \quad \text{and} \quad
(b_1,b_2,b_3) := (x_1,x_2,0),
\]
where $x_2 > 0$ and $x_3 > 0$. By hypothesis~(1), it follows that
\[
\widetilde{g}_J(x_1,x_2,x_3) \widetilde{g}_K(x_1, x_2, 0) \geq
\widetilde{g}_K(x_1, x_2, 0)^2,
\]
and taking limits as $x_2 = x_{K \setminus L} \to 0^+$ and
$x_3 = x_{J \setminus K} \to 0^+$, we have that
\[
\widetilde{g}_J(x_1,0,0) \widetilde{g}_K(x_1, 0, 0) \geq
\widetilde{g}_K(x_1, 0, 0)^2,
\]
and so (\ref{Eextension}) holds as required.

Finally, to show that $(3) \implies (2)$, given positive Hankel
matrices $A_1$, \dots, $A_m$ arising from moment sequences
in~$\momentr([0,1])$, let
\[
J := \{ l \in [m] : a_{l,11} > 0 \}
\qquad \text{and} \qquad
K := \{ l \in [m] : a_{l,22} > 0 \}.
\]
Note that $K \subset J \subset [m]$. Recalling that the only Hankel
matrices arising from~$\momentr([0,1])$ and having zero entries are of
the form~$H_{a \delta_0}$ for some $a \in [0, \rho)$, we may write
\begin{equation}\label{Elast}
F[A_1, \dots, A_m] = \left( g_J(a_{l,11} : l \in J) - g_K(a_{l,11} : l
\in K) \right) H_{\delta_0} + g_K[ A_l : l \in K ].
\end{equation}
For example, given $a$, $b$, $c$, $d > 0$, we have that
\begin{align*}
F \left[ \begin{pmatrix} a & b \\ b & c \end{pmatrix}, \ 
\begin{pmatrix} d & 0 \\ 0 & 0 \end{pmatrix}, \ 
\begin{pmatrix} 0 & 0 \\ 0 & 0 \end{pmatrix} \right]
= &\ \begin{pmatrix}
g_{\{ 1, 2 \}}(a,d) & g_{\{1\}}(b)\\ g_{\{1\}}(b) & g_{\{1\}}(c)
\end{pmatrix}\\
= &\ (g_{\{ 1, 2 \}}(a,d) - g_{\{1\}}(a)) \begin{pmatrix} 1 & 0\\ 0 &
0\end{pmatrix} + g_{\{1\}} \left[ \begin{pmatrix} a & b \\ b & c
\end{pmatrix} \right].
\end{align*}
The proof concludes by observing that both terms in the
right-hand side of (\ref{Elast}) are positive semidefinite, by the
Schur product theorem and hypothesis~(3):
\begin{align*}
g_J(a_{l,11} : l \in J) \geq \lim_{a_{l,11} \to 0^+\ \forall l \in J
\setminus K} g_J(a_{l,11} : l \in J) & = 
\widetilde{g}_J(a_{l,11} : l \in K) \\
 & \geq g_K(a_{l,11} : l \in K).
\qedhere
\end{align*}
\end{proof}

As Theorem~\ref{T1sided-multi} shows, the notion of facewise
absolutely monotone maps on $[0,\rho)^m$ is a refinement of absolute
monotonicity, emerging from the study of positivity preservers of
tuples of moment sequences, or, rather, of the Hankel matrices arising
from them. If, instead, one studies maps preserving positivity on
tuples of all positive semidefinite matrices, or even all Hankel
matrices, then this richer class of maps does not arise.

\begin{proposition}\label{P1sided-multi}
Suppose $\rho \in (0,\infty]$ and $F : I^m \to \R$, where $I = [0,\rho)$.
The following are equivalent.
\begin{enumerate}
\item $F[-]$ preserves positivity on the space of $m$-tuples of
positive Hankel matrices with entries in $I$.
\item $F$ is absolutely monotonic on $I^m$.
\item $F[-]$ preserves positivity on the space of $m$-tuples of all
positive matrices with entries in $I$.
\end{enumerate}
\end{proposition}

\begin{proof}
Clearly $(2) \implies (3) \implies (1)$. Now suppose (1) holds. By
Theorem~\ref{Thorn-hankel2}, $F$ is absolutely monotonic on the
domain~$(0,\rho)^m$, and agrees there with an analytic function
$g : D(0,\rho)^m \to \C$. We now claim $F \equiv g$ on $I^m$. The
proof is by induction on $m$, with the $m = 1$ case shown in
Proposition~\ref{P1sided}.

Suppose $m>1$, and let
$\bc = (c_1, \dots, c_m) \in I^m \setminus ( 0, \rho )^m$. Note that
at least one coordinate of~$\bc$ is zero. We choose
$\bu_n = (u_{1,n}, \dots, u_{m,n}) \in ( 0, \rho )^m$ such that
$\bu_n \to \bc$, and we wish to show that
$F( \bu_n ) = g( \bu_n ) \to F( \bc )$. Let
\[
H := \begin{pmatrix}
1 & 0 & 1\\
0 & 1 & 1\\
1 & 1 & 2
\end{pmatrix} \qquad \text{and} \qquad
A_{l,n} := \begin{cases}
u_{l,n} \bone_{3 \times 3} & \text{if } c_l > 0,\\
u_{l,n} H & \text{if } c_l = 0.
\end{cases}
\]
Using (1) and the induction hypothesis for the $(1,2)$ and~$(2,1)$
entries, it follows that
\[
\lim_{n \to \infty} F[A_{1,n}, \dots, A_{m,n}] =
\begin{pmatrix}
g(\bc) & F(\bc) & g(\bc)\\
F(\bc) & g(\bc) & g(\bc)\\
g(\bc) & g(\bc) & g(\bc)
\end{pmatrix}
\in \bp_3.
\]
Computing the determinants of the leading principal minors gives
\[
g(\bc) \geq 0, \qquad g(\bc) \geq |F(\bc)|, \qquad \text{and} \qquad
-g(\bc) (g(\bc) - F(\bc))^2 \geq 0.
\]
Hence $F(\bc) = g(\bc)$, and the proof is complete.
\end{proof}

\subsection{Transformers of moment-sequence tuples: the general case}

Having resolved the characterization problem for functions defined on
the positive orthant, we now work over the whole of $\R^m$. This
requires us to consider admissible measures which may have support
outside $[-1,1]$. For such measures, the mass no longer dominates all
moments, and so we include in our test sets truncations of the
corresponding moment sequences, whereas for measures supported in
$[-1,1]$, the full moment sequence lies in the test set. More
precisely, we have the following definition.

\begin{definition}
Given $K \subset \R$ and $\rho \in (0,\infty]$, let
$\widetilde{\momentr}(K)$ be the collection of
possibly truncated moment sequences for all measures
$\mu \in \meas^+(K)$, where each sequence is truncated prior to the
first moment of $\mu$ that lies outside
$(-\rho,\rho)$, or is not truncated if no such moment exists.
\end{definition}

Clearly, if $\rho = \infty$, then
$\widetilde{\momentr}(K) = \momentr(K)$, while if $\rho$ is finite and
$K \subset [-1,1]$, then $\widetilde{\momentr}(K) = \momentr(K)$.
Next is a more complex example, which occurs in the following theorem;
see particularly Step~5 of its proof.

\begin{example}\label{Etilde}
If $K = \{ -1, v, 1 \}$ for $v > 1$, and $\rho < \infty$, then
$\widetilde{\momentr}(K)$ consists of $\momentr(\{ -1, 1 \})$ together
with truncated moment sequences of measures with positive mass
at~$v$. If $\mu = a \delta_{-1} + b \delta_1 + c \delta_v$ with
$a$, $b \geq 0$ and $c > 0$, then the moments of $\mu$ are unbounded,
and $\widetilde{\momentr}(K)$ contains the truncated moment sequence
$( s_0(\mu), \ldots, s_{n - 1}(\mu) )$, where $n$ is the smallest
positive integer such that $| a ( -1 )^n + b + c v^n | \geq \rho$.
\end{example}

We can now state our final main result in this section.

\begin{theorem}\label{T2sided-multi}
Suppose $F : I^m \to \R$, where $m \geq 1$ and $I = (-\rho, \rho)$
with $0 < \rho \leq \infty$. The following are equivalent.
\begin{enumerate}
\item For some $\delta > 0$, the function $F$, when applied entrywise,
maps $\widetilde{\momentr}([-1,1+\delta])^m$ into the set of possibly
truncated moment sequences of measures on $\R$.

\item For every $\delta > 0$, the function $F$, when applied
entrywise, maps $\widetilde{\momentr}([-1,1+\delta])^m$ into the set
of possibly truncated moment sequences of measures on $\R$.
  
\item Applied entrywise, the function $F$ maps $\bp_N(I)^m$ into
$\bp_N(\R)$ for any $N \geq 1$.

\item The function $F$ is absolutely monotonic on $[ 0, \rho )^m$ and
agrees on $I^m$ with an analytic function.
\end{enumerate}
\end{theorem}

In particular, analogously to the one-variable case,
Theorem~\ref{T2sided-multi} strengthens the multivariable analogue of
Schoenberg's theorem in~\cite{fitzgerald} by using only Hankel
matrices arising from tuples of moment sequences. Moreover, akin to
the $m=1$ case, the proof reveals that one only requires Hankel
matrices of rank at most~$3$.

Given any $v > 0$, we let
$\momentr_v := \widetilde{\momentr}(\{ -1, v, 1 \})$
and
\[
\momentr_{[v]} :=
\bigcup_{s_1 \in \{ -1, 0, 1 \}, s_2 \in \{ -v, 0, v \}}
\momentr(\{ s_1, s_2 \}).
\]

\begin{corollary}\label{C2sided-multi}
The hypotheses in Theorem \ref{T2sided-multi} are also
equivalent to the following.
\begin{enumerate}
\setcounter{enumi}{4}
\item There exist $\epsilon > 0$ and $u_0 \in ( 0, 1 )$ such that
$F[-]$ maps 
\[
( \momentr_{[u_0]} )^m \cup
\bigcup_{v_1, \dots, v_m \in ( 0, 1 + \epsilon )} \momentr_{v_1}
\times \cdots \times \momentr_{v_m}
\]
into the set of possibly truncated moment sequences of
measures on $\R$.
\end{enumerate}
\end{corollary}

As the reader will observe, hypothesis~(5) is stronger, even in the
one-dimensional case, than the corresponding hypothesis in
Theorem~\ref{T2sided-general}. As the proof shows, these extra
assumptions are required to obtain continuity on every orthant and on
`walls' between orthants, as well as real analyticity on one-parameter
curves.

\begin{remark}
Theorem~\ref{T2sided-multi} is the only instance when we deviate from
Table~\ref{Tmultivar} in the Introduction, but it should not come as a
surprise that stronger conditions are required to guarantee real
analyticity in several variables.
\end{remark}

\begin{proof}[Proof of Theorem~\ref{T2sided-multi} and
Corollary~\ref{C2sided-multi}]

Clearly $(4) \implies (3) \implies (2) \implies (1) \implies (5)$ by
the Schur product theorem. Thus, we will assume (5) and obtain (4). By
Theorem~\ref{Thorn-hankel2}, the function~$F$ is absolutely monotonic
on the open positive orthant $(0,\rho)^m$, and equals the restriction
to $(0,\rho)^m$ of an analytic function $g : D(0,\rho)^m \to \C$. We
now show that $F \equiv g$ on all of~$(-\rho,\rho)^m$. The proof
follows the $m=1$ case in Section~\ref{Smomentsm11}; for ease of
exposition, we break it up into steps.

\noindent\textbf{Step 1.}
We first prove $F$ is locally bounded. This follows by using
$\momentr_2(\{ -1, 1 \})^m$, as in the proof of
Lemma~\ref{Lbounded}. As above, this gives that
\begin{equation}\label{Epartial2}
F[-] : (\momentr_{[u_0]})^m \cup \momentr_{v_1} \times \cdots
\times \momentr_{v_m} \to \moment([-1,1]) \qquad
\text{for all } v_1, \ldots, v_m \in (-1,1).
\end{equation}

\noindent\textbf{Step 2.}
Next, we show that $F$ is continuous on $(-\rho,\rho)^m$. The first
objective is to show continuity of~$F$ inside each open orthant
of~$(-\rho,\rho)^m$. Given non-zero scalars $c_1$, \ldots, $c_m$ with
$|c_1|$, \ldots, $|c_m| < \rho$, and any sequence
$\{ (v_{1,n}, \ldots, v_{m,n}) : n \geq 1 \} \subset \R^m$ converging
to the origin, let
\begin{equation}
a_{l,n} := |c_l| + \frac{\sgn(c_l) u_0}{u_0 - u_0^3} v_{l,n}
\quad \text{and} \quad
\mu_{l,n} := a_{l,n} \delta_{\sgn(c_l)} +
\frac{|v_{l,n}|}{u_0 - u_0^3} \delta_{-\sgn(v_{l,n}) u_0}
\end{equation}
for $l = 1, \ldots, m$. Note that, for all sufficiently large $n$, the
sequence $\bs(\mu_{l,n}) \in \momentr_{[u_0]}$.

We now follow the proof of Proposition~\ref{Pcont}. Suppose that
$F[H_{\mu_{1,n}}, \dots, H_{\mu_{m,n}}] = H_{\sigma_n}$ for some
admissible measure $\sigma_n \in \meas^+([-1,1])$, for
every~$n \geq 1$. The polynomials $p_\pm(t) := (1 \pm t)(1 - t^2)$
are non-negative on $[-1,1]$, so, by~(\ref{Etrick}),
\begin{multline}
\int_{-1}^1 p_\pm(t)\std\sigma_n \geq 0, \\
\text{so} \quad F \left( s_0(\mu_{l,n})_{l=1}^m \right)
- F \left(s_2(\mu_{l,n})_{l=1}^m \right) \geq
\left| F \left( s_1(\mu_{l,n})_{l=1}^m \right) -
F \left( s_3(\mu_{l,n})_{l=1}^m \right) \right|.\label{Etrick-multi}
\end{multline}
Computing the moments of $\mu_{l,n}$ gives the following:
\begin{align}\label{Emoments}
\begin{aligned}
s_0(\mu_{l,n}) & = |c_l| + \frac{\sgn(c_l) u_0 +
\sgn(v_{l,n})}{u_0 - u_0^3} v_{l,n}, \qquad \quad
s_1(\mu_{l,n}) = c_l, \\
s_2(\mu_{l,n}) & = |c_l| + \frac{\sgn(c_l) u_0 + 
\sgn(v_{l,n}) u_0^2}{u_0 - u_0^3} v_{l,n}, \qquad
s_3(\mu_{l,n}) = c_l + v_{l,n}.
\end{aligned}
\end{align}
As $n \to \infty$, by the continuity of $F$ in~$(0,\rho)^m$,
the left-hand side of~(\ref{Etrick-multi}) goes to zero, whence so
does the right-hand side, which is
$|F(c_1, \dots, c_m) - F(c_1 + v_{1,n}, \dots, c_m + v_{m,n})|$. This
proves the continuity of~$F$ at~$(c_1, \dots, c_m)$, so in every open
orthant of $(-\rho,\rho)^m$.

To conclude this step, we show $F$ is continuous on the boundary of
the orthants, that is, on the union of the coordinate hyperplanes:
\[
\cZ := \{ (x_1,\dots,x_m) \in (-\rho,\rho)^m : x_1 \cdots x_m = 0 \}.
\]
The proof is by induction on~$m$, with the case $m=1$ shown in
Proposition~\ref{Pcont}. For general $m \geq 2$, by the induction
hypothesis $F$ is continuous when restricted to~$\cZ$. It remains
to prove~$F$ is continuous at a point
$\bc = (c_1, \dots, c_m) \in \cZ$ when approached along a sequence
$\{ (c_1 + v_{1,n}, \dots, c_m + v_{m,n}) : n \geq 1 \}$ which lies in
the interior of some orthant in~$(-\rho,\rho)^m$. Repeating the
computations for (\ref{Emoments}), with the same sequences $a_{l,n}$ and
$\mu_{l,n}$, and the polynomials $p_\pm(t)$, we note that if
$c_l \neq 0$ then $s_0(\mu_{l,n}) > 0$ and $s_2(\mu_{l,n}) > 0$ for
all sufficiently large $n$, while if $c_l = 0$ then
$s_0(\mu_{l,n}) > 0$ and $s_2(\mu_{l,n}) > 0$ for all $n$, since
$c_l + v_{l,n} \neq 0$ by assumption. Therefore, in all cases, the
left-hand side of~(\ref{Etrick-multi}) eventually equals
$F(\bu_n) - F(\bu'_n)$, with $\bu_n$ and $\bu'_n$ in the positive open
orthant~$(0,\rho)^m$, and both converging
to~$|\bc| := (|c_1|, \dots, |c_m|)$. Since $F \equiv g$
on~$(0,\rho)^m$ for some analytic function $g$ on
$D(0,\rho)^m$, so~(\ref{Etrick-multi}) gives that
\[
\lim_{n \to \infty} |F(\bc) - F(c_1 + v_{1,n}, \dots, c_m + v_{m,n})|
\leq \lim_{n \to \infty} F(\bu_n) - F(\bu'_n) = 
g(|\bc|) - g(|\bc|) = 0.
\]
It follows that $F$ is continuous at all $\bc \in \cZ$, and hence on
all of~$(-\rho,\rho)^m$, as claimed.

\noindent\textbf{Step 3.}
The next step in the proof is to show that it suffices to consider $F$ to
be smooth. This is achieved using a mollifier argument, exactly as in the
one-variable situation.

\noindent\textbf{Step 4.}
Henceforth we assume $F$ is smooth on $(-\rho,\rho)^m$; akin
to the one-variable case, we will show that $F$ is in fact real analytic.
The proof extends across multiple steps below. The first step is encoded
into the following technical lemma, for convenience.

\begin{lemma}\label{Leta}
Fix $\rho \in (0,\infty]$ and a non-zero vector $\bv \in \R^m$. For
any $\bc \in (-\rho,\rho)^m$, let
\begin{equation}\label{Eetac}
\eta_{\bv,\bc} := \begin{cases}
e^{-\| \bv \|_\infty} \qquad & \text{if } \rho = \infty,\\
e^{-\| \bv \|_\infty} (\rho - \| \bc \|_\infty) &
\text{if } \rho < \infty.
\end{cases}
\end{equation}
Then, for any $\bw \in (-\rho,\rho)^m$, there exists
$\bc \in (-\rho,\rho)^m$ such that $\bw = \bc + \eta_{\bv,\bc} \bone$,
where the vector $\bone := ( 1, \ldots, 1 )$.
\end{lemma}

\begin{proof}
The assertion is immediate if $\rho = \infty$, so we suppose henceforth
that $\rho$ is finite. Let
\[
g( t ) := \| \bw - t \bone \|_\infty -
( \rho - t e^{\| \bv \|_\infty}) \qquad ( t \geq 0 ).
\]
Clearly $g(0) < 0 < g(\rho)$, so $g$ has a root $t_0 \in ( 0, \rho )$.
Now the vector $\bc := \bw - t_0 \bone$ is as required
(and $t_0 = \eta_{\bv,\bc}$).
\end{proof}

\noindent\textbf{Step 5.}
We now claim that for every $\bc \in (-\rho,\rho)^m$ and every unit
direction vector $\bv = (v_1, \dots, v_m) \in S^{m-1}$, the function $F$
is real analytic in the one-parameter space
\[
\{ \bc + \eta_{\bv,\bc} e^{-x \bv} : x \in (-1,1) \} \subset
(-\rho,\rho)^m,
\]
at the point $x = 0$, i.e., at ${\bf w} = \bc + \eta_{\bv,\bc} \bone$.
Here $\eta_{\bv,\bc}$ and $\bone$ are as in Lemma~\ref{Leta}, and we also
use the notation
\[
e^{-x\bv} := (e^{-x v_1}, \dots, e^{-x v_m}).
\]
Notice moreover that the $l$th coordinate of
$\bc + \eta_{\bv,\bc} e^{-x \bv}$ is strictly bounded above in
absolute value by
$\| \bc \|_\infty + \eta_{\bv,\bc} e^{\| \bv \|_\infty}$, which is no
more than $\rho$.

To show the claim, we use the notation
$|\bc| := (|c_1|, \ldots, |c_m|)$ and also fix a scalar
$x \in (-1,1)$. We let $p_{\pm,n}(t) := (1 \pm t)(1-t^2)^n$ for
$n \geq 0$ and
\[
\mu_{l,s} := |c_l| \delta_{\sgn(c_l)} 
+ \eta_{\bv,\bc} e^{-x v_l} \delta_{e^{-s v_l}} \qquad
\text{whenever } 0 < s < ( 1 - x ) / ( 2 n + 1 ),
\]
where $1 \leq l \leq m$.

As $p_{\pm,n}(t) \geq 0$ for all $t \in [-1,1]$ and all $n \geq 0$,
applying (\ref{Etrick}) gives that
\[
\left| \sum_{k=0}^n \binom{n}{k} (-1)^k F(|\bc| +
\eta_{\bv,|\bc|} e^{-(x + 2ks) \bv}) \right| \!\geq\! \left|
\sum_{k=0}^n \binom{n}{k} (-1)^{n-k} F(\bc + \eta_{\bv,\bc}
e^{-(x + (2k+1)s)\bv}) \right|,
\]
where we note that $\eta_{\bv,\bc} = \eta_{\bv,|\bc|}$, and that
all arguments of $F$ lie in $(-\rho, \rho)^m$ by the restriction
on~$s$. Note that we use the fact that our test set contains
$\widetilde{\momentr}(\{ -1, v, 1 \})$ for $v \in (1,1+\epsilon)$
here, and only here, in this proof.

Now setting $H_{\bv,\bc}(x) := F(\bc + \eta_{\bv,\bc} e^{-x \bv})$,
dividing both sides of this inequality by $s^n$,
and then taking~$s \to 0^+$, it follows that
\[
\left| \frac{\std^n}{\std x^n} H_{\bv,|\bc|}^{(n)}(x) \right|
\geq
\left| \frac{\std^n}{\std x^n} H_{\bv,\bc}^{(n)}(x) \right|
\qquad \text{whenever } x \in (-1,1).
\]
These estimates prove that the function $F$ is real analytic at the
point in the one-parameter space as claimed.

\noindent\textbf{Step 6.}
We now complete the proof. The real-analytic local diffeomorphism
\[
T : (u_1,\cdots,u_m) \mapsto (e^{u_1}-1, e^{u_2}-1, \cdots, e^{u_m}-1)
\]
maps the origin to itself and, by the previous step, the function
\[
\bu \mapsto F( \bc + \eta_{\bv,\bc} \bone + \eta_{\bv,\bc} T(-\bu))
\]
is smooth and real analytic in the unit ball along every straight line
passing through the origin. Standard criteria for real analyticity
(see \cite[Theorem~5.5.33]{Salah}, for example) now give that~$F$ is
real analytic at the point $\bc + \eta_{\bv,\bc} \bone$, hence at
every point $\bw \in (-\rho,\rho)^m$, by Lemma~\ref{Leta}.

Finally, recall that $F$ agrees on $(0,\rho)^m$ with an analytic
function $g : D(0,\rho)^m \to \C$. As $F : (-\rho,\rho)^m \to \R$ is
real analytic, so $F = g|_{(-\rho,\rho)^m}$ and the proof is complete.
\end{proof}

\begin{remark}
As Step $2$ in the proof above shows, we may replace
$(\momentr_{[u_0]})^m$ in hypothesis~(5) of
Corollary~\ref{C2sided-multi} by
$\momentr_{[u_1]} \times \cdots \times \momentr_{[u_m]}$ for any
$u_1$, \ldots, $u_m \in (0,1)$.
\end{remark}

\begin{remark}
Akin to the one-dimensional case, one may now show that
Theorems~\ref{T1sided-multi} and~\ref{T2sided-multi} hold more
generally for tuples of measures with bounded mass. More precisely,
one should fix $\rho_1$, \dots, $\rho_m \in (0,\infty)$
and work with tuples of admissible measures $(\mu_1, \dots, \mu_m)$
supported in~$[-1,1]$ and such that
$s_0(\mu_l) <\rho_l$ for $l = 1$, \ldots, $m$, whence
$s_k(\mu_l) < \rho_l$ for every $k \geq 0$ and all such~$l$. As
discussed in the Introduction, this explains how our results unify and
strengthen the Schoenberg--Rudin theorem and the
FitzGerald--Micchelli--Pinkus result for positivity preservers.

To prove Theorem~\ref{T1sided-multi} for
$F : I_1 \times \dots \times I_m \to \R$, where $I_l = [0,\rho_l)$,
one should first define facewise absolutely monotonic maps on
$I_1 \times \dots \times I_m$ using the relative interiors of the faces
cut out by the same functionals as for~$[0,\rho)^m$. The
existing proof for the case $\rho_1 = \cdots = \rho_m$ goes
through with minimal modifications, including to
Theorem~\ref{Thorn-hankel2}. The same is true for proving
Theorem~\ref{T2sided-multi} with the domain $(-\rho_1,\rho_1) \times
\cdots \times (-\rho_m, \rho_m)$ in place of $(-\rho,\rho)^m$.
\end{remark}

\begin{remark}
There is a simple and potentially very useful conditioning operation
which can assist with numerical or computational entrywise
manipulation of Hankel matrices or Hankel kernels arising from
moments. Namely, the moments
\[
s_\alpha = \int_K x^\alpha \std\mu( x ) \qquad ( \alpha \in \nnZ^m)
\]
of a positive measure with compact support $K$ can be rescaled,
\[
s_\alpha \mapsto u_\alpha = t^{|\alpha|} s_\alpha,
\]
by a factor $t > 0$, so that $u_\alpha$ are the moments of a positive
measure supported by the unit cube, or even by its interior. Of
course, \textit{a priori} information on the size of the support $K$
is essential for this step, but in this way some of the complications
outlined in Theorem~\ref{T2sided-multi} and its proof can be avoided.
\end{remark}

\section{Laplace-transform interpretations}\label{SLaplace}

When speaking about completely monotonic or absolutely monotonic
functions one cannot leave aside Laplace transforms. We briefly touch
the subject below, in connection with our theme.

Let $F$ be an absolutely monotonic function on~$(0,\infty)$, and let
$\mu$ and $\sigma$ be admissible measures supported on~$[0,1]$ such that
\begin{equation}\label{Emomentpreserving}
F(s_k(\mu)) = s_k(\sigma) \qquad \text{for all } k \geq 0.
\end{equation}
By the change of variable $x = e^{-t}$, we can push forward the
restriction of the measure~$\mu$ to~$(0,1]$ to a measure $\mu_1$
on~$[0,\infty)$, and similarly for~$\sigma$. Thus, with the possible
loss of zeroth-order moments, we obtain
\[
s_k(\mu) = \int_0^\infty e^{-kt} \std\mu_1(t) \quad \text{and} \quad
s_k(\sigma) = \int_0^\infty e^{-kt} \std\sigma_1(t).
\]
If $\cL$ denotes the Laplace transform, so that
\[
\cL\nu(z) = \int_0^\infty e^{-tz} \std\nu(t),
\]
then $\cL \nu$ is a complex-analytic function in the open half-plane
$\C^+ := \{ z \in \C: \Re z > 0 \}$. Our
assumption~(\ref{Emomentpreserving}) becomes
\[
F(\cL\mu_1(k)) = \cL\sigma_1(k) \qquad
\text{for all } k \geq 1,
\]
and a classical observation due to Carlson \cite{Carlson} implies that
\[
F(\cL\mu_1(z)) =  \cL\sigma_1(z) \qquad
\text{for all } z \in \C^+.
\]
More precisely, Carlson's Theorem asserts that a bounded
analytic function in the right half-plane is
identically zero if it vanishes at all positive integers. The proof
relies on the Phragm\'en--Lindel\"of principle \cite{Phragmen-Lindelof};
see also \cite{Boas-book} or \cite[\S5.8]{Titchmarsh} for more details.

In this section, we will show some results from the interplay between
the Laplace transform and functions which transform positive Hankel
matrices.

For point masses, the situation is rather straightforward. If
$\mu = \delta_{e^{-a}}$ for some point~$a \in [0, \infty)$,
and~$F(x) = \sum_{n=0}^\infty c_n x^n$, then
\[
F(\cL \mu(z)) = F (e^{-az}) = \sum_{n=0}^\infty c_n e^{-anz} = 
\cL \sigma_1(z),
\]
where 
\[
\sigma_1 = \sum_{n=0}^\infty c_n \delta_{an} \qquad \text{and} \qquad
\sigma = \sum_{n=0}^\infty c_n \delta_{e^{-an}}.
\]
More generally, if $\mu$ has countable support, then the transform
$F[-]$ will yield a measure with countable support also. A strong
converse to this is the following result.

\begin{proposition}
Let $a \in (0,1)$ and suppose the function
$F : x \mapsto \sum_{n=0}^\infty c_n x^n$ is absolutely monotonic
on~$(0,\infty)$. The following are equivalent.
\begin{enumerate} 
\item There exists an admissible measure~$\mu$ on~$[0,1]$ such that
\[
F(s_k(\mu)) = a^k \qquad \text{for all } k \geq 0.
\]
\item $F(x) = x^N$ for some $N \geq 1$, and $\mu = \delta_{a^{1/N}}$.
\end{enumerate}
\end{proposition}

\begin{proof}
That $(2) \implies (1)$ is clear. Now suppose $(1)$ holds. Setting
$\psi(t) := {-\log t}$,
\[
s_k(\mu) = \int_0^1 x^k\std\mu(t) =
\int_0^\infty e^{-kt} \std \nu(t) = \cL\nu(k) \qquad
\text{for all } k \geq 0,
\]
where $\nu := \psi_* \mu$ is the push-forward of $\mu$ under~$\psi$.
If $a = e^{-\lambda}$ for some $\lambda > 0$, then, by
assumption, 
\[
F(\cL \nu(k)) = e^{-\lambda k} \qquad \text{for all } k \geq 1.
\]
and, by Carlson's Theorem,
\begin{equation}\label{Elaplace}
F(\cL \nu(z)) = e^{-\lambda z} \qquad \text{for all } z \in \C^+.
\end{equation}
In view of Bernstein's theorem, Theorem~\ref{Tbernstein}, the
function $\cL \nu$ is completely monotonic on $[0,\infty)$. Now, since
the composition of an absolutely monotonic function and a completely
monotonic function is completely monotonic, so
\[
z \mapsto (\cL\nu(z))^k = \left(\int_0^\infty e^{-zt}\std\nu(t)\right)^k
\] 
is completely monotonic on $[0,\infty)$ for all $k \in \nnZ$.
Thus, by another application of Bernstein's theorem,
there exists an admissible measure $\nu_k$ on $[0,\infty)$ such that
\[
(\cL\nu(z))^k = \int_0^\infty e^{-zt}\std\nu_k(t) \qquad
\text{for all } z \in \C^+.
\]
Using the above expression, we can rewrite (\ref{Elaplace}) as
\[
F(\cL \nu(z)) = \sum_{n=0}^\infty c_n (\cL \nu_n)(z) =
\cL\left(\sum_{n=0}^\infty c_n \nu_n\right)(z) = e^{-\lambda z} =
(\cL \delta_\lambda)(z),
\]
and, by the uniqueness principle for Laplace transforms, we conclude
that
\[
\sum_{n=0}^\infty c_n \nu_n = \delta_\lambda.
\]
Now, let $A$ be any measurable subset of $[0,\infty)$ that does not
contain $\lambda$. Then,
\[
\left(\sum_{n=0}^\infty c_n \nu_n\right)(A) = \delta_\lambda(A) = 0.
\]
Since $c_n \geq 0$, it follows that $c_n \nu_n(A) = 0$ for all
measurable sets $A$ not containing $\lambda$, and
all~$n \in \nnZ$. Hence, either $c_n = 0$, or
$\nu_n = \delta_\lambda$. Moreover, $\sum_{n=0}^\infty c_n = 1$.

Now, suppose $c_n \neq 0$ for some $n$. By the above argument, we must
have $\nu_n = \delta_\lambda$. Thus,  
\[
\cL\nu_n(z) = \left(\int_0^\infty e^{-zt}\std\nu(t)\right)^n = 
e^{-\lambda z} \qquad \text{for all } z \in \C^+.
\]
Equivalently, 
\[
\int_0^\infty e^{-zt}\std\nu(t) = e^{-\lambda z / n},
\]
and applying the uniqueness principle for the Laplace transform one
more time gives that $\nu = \delta_{\lambda/n}$. Hence $c_n \neq 0$
for at most one $n$, say for $n=N$, so $F(x) = x^N$ and
$\nu = \delta_{\lambda /N}$. Finally, since~$\nu = \psi_* \mu$,
we conclude that $\mu = \delta_{a^{1/N}}$, as claimed.
\end{proof}

\appendix

\section{Two lemmas on adjugate matrices}\label{Ap}

In this appendix we prove two lemmas. These allow us
to establish Equation~(\ref{Emiracle}), which is key to our proof of
Theorem~\ref{TpolyTN}, and they may be of independent interest.

Let $\F$ denote an arbitrary field. Given a matrix
$M \in \F^{N \times N}$, where $N \geq 1$, and a function
$f: \F \to \F$, we let $\adj( M )$ denote the adjugate matrix of $M$
and $f[ M ] \in \F^{N \times N}$ the matrix obtained by applying $f$
to each entry of $M$.

\begin{lemma}\label{L1}
Given a polynomial
$f( x ) =
\alpha_0 + \alpha_1 x + \cdots + \alpha_n x^n + \cdots \in \F[ x ]$
and a matrix $M \in \F^{N \times N}$, the polynomial
\[
\det f[ x M ] = \alpha_0 \, \alpha_1^{N - 1} \,
\bone_{1 \times N} \adj( M ) \bone_{N \times 1} \, x^{N - 1} +
O( x^N ).
\]
\end{lemma}

\begin{proof}
Let $M$ have columns $\bm_1$, \dots, $\bm_N$; we write
$M = ( \bm_1 | \cdots | \bm_N )$ to denote this. Using the
multi-linearity of the determinant, we see that
\begin{equation}\label{Edet}
\det f[ x M ] = \sum_{i_1, \dots, i_N = 0}^\infty
\alpha_{i_1} \cdots \alpha_{i_N} \, x^{i_1 + \cdots + i_N} 
\det( \bm_1^{\circ i_1} \mid \dots \mid \bm_N^{\circ i_N} ).
\end{equation}
Observe that the only way to obtain a term where $x$ has degree less
than $N - 1$ is for at least two of the indices $i_l$ to be $0$. The
corresponding determinants are all $0$ since they contain two columns
equal to $\bone_{N \times 1}$.

For terms containing $x^{N - 1}$, the only ones where
the determinant does not contain two columns equal to
$\bone_{N \times 1}$ sum to give
\[
\alpha_0 \, \alpha_1^{N - 1} \, x^{N - 1} \sum_{l = 1}^N
\det( \bm_1 \mid \dots \mid 
\bm_{l - 1} \mid \bone_{N \times 1} \mid 
\bm_{l + 1} \mid \dots \mid \bm_N ).
\]
By Cramer's Rule, this sum is precisely
$\bone_{N \times 1}^T \adj( M ) \bone_{N \times 1}$.
\end{proof}

We also require the following result, which we believe to be
folklore. We include a proof for completeness.

\begin{lemma}\label{L2}
Suppose $M \in \F^{N \times N}$ has rank $N - 1$. If $\bu$ spans the
null space of $M^T$, and~$\bv$ spans the null space of $M$, then
$A = \adj M$ is a non-zero scalar multiple of~$\bv \bu^T$.
\end{lemma}

\begin{proof}
That $A \neq 0$ follows by considering the rank of
$M$. Since $\det M = 0$, we have that $A M = 0$ and
$M A = 0$. After taking the transpose, the first identity implies
that the rows of $A$ are multiples of $\bu^T$; the second identity
implies immediately that the columns of $A$ are multiples of $\bv$.
This gives the result.
\end{proof}

We may now show that
$\det M_4 =
-57168 \, \alpha_0 \, \alpha_1^2 \, \alpha_2 \, x^4 + O(  x^5 )$,
where
\[
M_4 := \sum_{k=0}^4 \alpha_k x^k M^{\circ k} \quad \text{and} \quad
M := \begin{pmatrix}
3 & 6 & 14 & 36 \\
6 & 14 & 36 & 98 \\
14 & 36 & 98 & 276 \\
36 & 98 & 284 & 842
\end{pmatrix}.
\]
Note that these matrices are totally non-negative, and would be Hankel
but for one entry.

By Lemma \ref{L1}, $\det M_4$ has no constant, linear, or quadratic term.
Moreover, since the matrix $M$ has rank $3$ and the vectors
\[
\bv = ( 6, -11, 6, -1 ) \quad \text{and}
\quad \bu = ( 46, -59, 18, -1 )
\]
span the null spaces of $M$ and $M^T$, respectively, Lemma~\ref{L2}
gives that $\adj( M )$ is equal to $c \bv \bu^T$ for
some non-zero $c \in \R$. The cubic term in $\det M_4$ equals
$c \, \bone^T \bv \, \bu^T \bone \, \alpha_0 \, \alpha_1^3 \, x^3$,
by Lemma~\ref{L1}, and this vanishes because $\bone^T \bv = 0$.

Finally, we compute the coefficient of the quartic term; we need to
examine all the terms in~(\ref{Edet}) that arise from quadruples
$(i_1, i_2, i_3, i_4 )$ which sum to~$4$. Terms with indices of the form
$( 4, 0, 0, 0 )$, $( 3, 1, 0, 0 )$, and $( 2, 2, 0, 0)$, and their
permutations, are zero since the determinants contain
two identical columns. We are therefore left with
quadruples of the form $( 2, 1, 1, 0 )$ and $( 1, 1, 1, 1 )$. The
term corresponding to $( 1, 1, 1, 1 )$ is zero since $\det M = 0$, as
one can see as $M$ does not have full rank.
Thus the only non-zero quartic terms arise from one
of the twelve permutations of the quadruple $( 2, 1, 1, 0 )$. Therefore
$\det M_4 = k \, \alpha_0 \, \alpha_1^2 \, \alpha_2 \, x^4 + O( x^5 )$,
and to find the constant $k$, we compute all twelve determinants.

\begin{center}
\begin{tabular}{|c|c||c|c|}
\hline
$i_1, i_2, i_3, i_4$ &
$\det( \bm_1^{\circ i_1} | \bm_2^{\circ i_2} | \bm_3^{\circ i_3} |
\bm_4^{\circ i_4} )$ &
$i_1, i_2, i_3, i_4$ &
$\det( \bm_1^{\circ i_1} | \bm_2^{\circ i_2} | \bm_3^{\circ i_3} |
\bm_4^{\circ i_4} )$ \\ \hline
$0,1,1,2$ & $1398912$ & $1,1,2,0$ & $-72224$\\
$0,1,2,1$ & $-138048$ & $1,2,0,1$ & $-46224$\\
$0,2,1,1$ & $-96384$ & $1,2,1,0$ & $21520$\\
$1,0,1,2$ & $-2431744$ & $2,0,1,1$ & $14432$\\
$1,0,2,1$ & $598304$ & $2,1,0,1$ & $-5208$\\
$1,1,0,2$ & $699552$ & $2,1,1,0$ & $-56$ \\ \hline
\end{tabular}
\end{center}
The sum of these determinants is $-57168$, as claimed. \qed




\begin{thebibliography}{88}

\bibitem{Akhiezer}
N.I.~Akhiezer.
\newblock {\em The classical moment problem and some related questions
in analysis}.
\newblock Translated by N.~Kemmer. Hafner Publishing Co., New York,
1965.

\bibitem{Salah}
S.~Baouendi, P.~Ebenfelt, and L.P.~Rothschild.
\newblock {\em Real submanifolds in complex space and their mappings},
volume~47 of {\em Princeton Mathematical Series}.
\newblock Princeton University Press, Princeton, 1999. 

\bibitem{BGKP-fixeddim}
A.~Belton, D.~Guillot, A.~Khare, and M.~Putinar.
\newblock Matrix positivity preservers in fixed dimension. I.
\newblock \href{http://dx.doi.org/10.1016/j.aim.2016.04.016}%
{\em Adv. Math.}, 298:325--368, 2016.

\bibitem{BGKP-survey}
A.~Belton, D.~Guillot, A.~Khare, and M.~Putinar.
\newblock {\em A panorama of positivity}.\hfill\break
\newblock Part I: Dimension free,
\href{https://www.springer.com/gp/book/9783030146399}{\em Analysis of
Operators on Function Spaces} (The Serguei Shimorin Memorial Volume;
A.~Aleman, H.~Hedenmalm, D.~Khavinson, M.~Putinar, Eds.), pp.~117--165;
Trends in Mathematics, Birkh\"auser, Chem, 2019;
\newblock Part II: Fixed dimension, {\em Complex Analysis and Spectral
Theory}, Proceedings of the CRM Workshop held at Laval University, QC,
May 21-25, 2018 (G.~Dales, D.~Khavinson, J.~Mashreghi, Eds.). CRM
Proceedings -- \href{http://dx.doi.org/10.1090/conm/743/14958}{AMS
Contemporary Mathematics} 743, pp.~109--150, American Mathematical
Society, 2020.
\newblock Parts 1 and 2 (unified) available at
\href{http://arxiv.org/abs/1812.05482}{arXiv:math.CA/1812.05482}.

\bibitem{BGKP-TN}
A.~Belton, D.~Guillot, A.~Khare, and M.~Putinar.
\newblock Totally positive kernels, P\'olya frequency functions,
and their transforms.
\newblock {\em Preprint},
\href{http://arxiv.org/abs/2006.16213}{arXiv:2006.16213}, 2020.

\bibitem{BCR}
C. Berg, J.P.R. Christensen, and P. Ressel.
\newblock \href{http://dx.doi.org/10.1007/978-1-4612-1128-0}%
{\em Harmonic analysis on semigroups. Theory of positive definite and
related functions.}
\newblock Graduate Texts in Mathematics, 100. Springer-Verlag, New York,
1984.

\bibitem{Bernstein}
S.~Bernstein.
\newblock Sur les fonctions absolument monotones.
\newblock \href{http://dx.doi.org/10.1007/BF02592679}{\em Acta Math.},
52(1):1--66, 1929.
  
\bibitem{Boas-book}
R.P.~Boas, Jr.
\newblock {\em Entire functions}.
\newblock Academic Press Inc., New York, 1954.

\bibitem{Bochner-pd}
S.~Bochner.
\newblock Hilbert distances and positive definite functions.
\newblock \href{http://www.jstor.org/stable/1969252}%
{\em Ann. of Math.~(2)}, 42:647--656, 1941.
  
\bibitem{Bochner-book}
S.~Bochner.
\newblock {\em Harmonic analysis and the theory of probability}.
\newblock University of California Press, Berkeley and Los Angeles, 1955.

\bibitem{Carlson}
F.~Carlson.
\newblock {\em Sur une classe des s\'eries de Taylor}.
\newblock Thesis (Ph.D.), Uppsala, 1914.

\bibitem{Donoghue}
W.F.~Donoghue.
\newblock \href{http://dx.doi.org/10.1007/978-3-642-65755-9}{\em Monotone
matrix functions and analytic continuation.}
\newblock Die Grundlehren der mathematischen Wissenschaften, Band 207.
Springer-Verlag, New York-Heidelberg, 1974.

\bibitem{FJ}
S.~Fallat and C.R.~Johnson.
\newblock{\em Totally nonnegative matrices}.
\newblock{Princeton University Press}, Princeton, 2011.

\bibitem{FJS}
S.~Fallat, C.R.~Johnson, and A.D.~Sokal.
\newblock Total positivity of sums, Hadamard products and Hadamard
powers: Results and counterexamples.
\newblock \href{http://dx.doi.org/10.1016/j.laa.2017.01.013}{\em Linear
Algebra Appl.}, 520:242--259, 2017;
Corrigendum, \href{http://dx.doi.org/10.1016/j.laa.2020.12.019}{\em
Linear Algebra Appl.}, 613:393--396, 2021.

\bibitem{FitzHorn}
C.H.~Fitz{G}erald and R.A.~Horn.
\newblock On fractional {H}adamard powers of positive definite matrices.
\newblock \href{http://dx.doi.org/10.1016/0022-247X(77)90167-6}%
{\em J. Math. Anal. Appl.}, 61(3):633--642, 1977.

\bibitem{fitzgerald}
C.H.~Fitz{G}erald, C.A.~Micchelli, and A.~Pinkus.
\newblock Functions that preserve families of positive semidefinite
matrices.
\newblock \href{http://dx.doi.org/10.1016/0024-3795(93)00232-O}%
{\em Linear Algebra Appl.}, 221:83--102, 1995.

\bibitem{Gantmacher_Vol2}
F.R.~Gantmacher.
\newblock {\em The theory of matrices. {V}ols. 1, 2}.
\newblock Translated by K.A.~Hirsch. Chelsea Publishing Co., New York, 1959.

\bibitem{GK}
F.R.~~Gantmacher and M.G.~Krein.
\newblock Sur les matrices compl\`etement non n\'egatives et
oscillatoires.
\newblock \href{http://www.numdam.org/item?id=CM_1937__4__445_0}{\em
Compos. Math.}, 4:445--476, 1937.

\bibitem{GK1}
F.R.~Gantmacher and M.G.~Krein.
\newblock {\em Oscillation matrices and kernels and small vibrations of
mechanical systems}.
\newblock Translated by A. Eremenko. Chelsea Publ. Co. New York, 2002.

\bibitem{GKR-crit-2sided}
D.~Guillot, A.~Khare, and B.~Rajaratnam.
\newblock Complete characterization of {H}adamard powers preserving {L}oewner
positivity, monotonicity, and convexity.
\newblock \href{http://dx.doi.org/10.1016/j.jmaa.2014.12.048}%
{\em J.~Math.~Anal.~Appl.}, 425(1):489--507, 2015.

\bibitem{GKR-lowrank}
D.~Guillot, A.~Khare, and B.~Rajaratnam.
\newblock Preserving positivity for rank-constrained matrices.
\newblock \href{http://dx.doi.org/10.1090/tran/6826}{\em Trans. Amer.
Math. Soc.}, 369(9):6105--6145, 2017.

\bibitem{Guillot_Rajaratnam2012b}
D.~Guillot and B.~Rajaratnam.
\newblock Functions preserving positive definiteness for sparse
matrices.
\newblock \href{http://dx.doi.org/10.1090/S0002-9947-2014-06183-7}%
{\em Trans. Amer. Math. Soc.}, 367:627--649, 2015.

\bibitem{HKKR}
H.~Helson, J.-P.~Kahane, Y.~Katznelson, and W.~Rudin.
\newblock The functions which operate on Fourier transforms.
\newblock \href{http://link.springer.com/article/10.1007/BF02559571}%
{\em Acta Math.}, 102(1):135--157, 1959.

\bibitem{Matrix01psd}
D.~Hershkowitz, M.~Neumann, and H.~Schneider.
\newblock Hermitian positive semidefinite matrices whose entries are
{$0$} or {$1$} in modulus.
\newblock \href{http://dx.doi.org/10.1080/03081089908818619}%
{\em Linear Multilinear Algebra}, 46(4):259--264, 1999.

\bibitem{horn}
R.A.~Horn.
\newblock The theory of infinitely divisible matrices and kernels.
\newblock \href{http://dx.doi.org/10.1090/S0002-9947-1969-0264736-5}%
{\em Trans. Amer. Math. Soc.}, 136:269--286, 1969.
  
\bibitem{Horn-toeplitz}
R.A.~Horn.
\newblock Infinitely divisible positive definite sequences.
\newblock \href{https://dx.doi.org/10.1090/S0002-9947-1969-0499938-6}%
{\em Trans. Amer. Math. Soc.}, 136:287--303, 1969.

\bibitem{Jain}
T.~Jain.
\newblock Hadamard powers of some positive matrices.
\newblock \href{http://dx.doi.org/10.1016/j.laa.2016.06.030}%
{\em Linear Algebra Appl.}, 528:147--158, 2017.

\bibitem{Kahane-Rudin}
J.-P.~Kahane and W.~Rudin.
\newblock Caract\'erisation des fonctions qui op\`erent sur les
coefficients de {F}ourier-{S}tieltjes.
\newblock {\em C. R. Acad. Sci. Paris}, 247:773--775, 1958.

\bibitem{Kh}
A.~Khare.
\newblock Smooth entrywise positivity preservers, a {H}orn--{L}oewner
master theorem, and symmetric function identities.
\newblock {\em Preprint},
\href{http://arxiv.org/abs/1809.01823}{arXiv:1809.01823}, 2018.

\bibitem{KT}
A.~Khare and T.~Tao.
\newblock On the sign patterns of entrywise positivity preservers in
fixed dimension.
\newblock {\em Amer.\ J.\ Math.}, in press,
\href{http://arxiv.org/abs/1708.05197}{arXiv:1708.05197}.

\bibitem{Lasserre}
J.B.~Lasserre.
\newblock \href{https://doi.org/10.1142/p665}{\em Moments, positive
polynomials and their applications}.
\newblock Imperial College Press Optimization Series, 1. Imperial
College Press, London, 2010.

\bibitem{Lorch-Newman}
L.~Lorch and D.J.~Newman.
\newblock On the composition of completely monotonic functions and
completely monotonic sequences and related questions.
\newblock
\href{http://jlms.oxfordjournals.org/content/s2-28/1/31.extract}%
{\em J. London Math. Soc. (2)}, 28(1):31--45, 1983.

\bibitem{Loewner34}
K.~L{\"o}wner.
\newblock \"{U}ber monotone {M}atrixfunktionen.
\newblock \href{http://dx.doi.org/10.1007/BF01170633}{\em Math. Z.},
  38(1):177--216, 1934.

\bibitem{Phragmen-Lindelof}
E.~Phragm{\'e}n and E.~Lindel{\"o}f.
\newblock Sur une extension d'un principe classique de l'analyse et
sur quelques propri\'et\'es des fonctions monog\`enes dans le
voisinage d'un point singulier.
\newblock \href{http://link.springer.com/article/10.1007/BF02415450}%
{\em Acta Math.}, 31(1):381--406, 1908.

\bibitem{Peller}
V.~Peller.
\newblock \href{http://dx.doi.org/10.1007/978-0-387-21681-2}{\em Hankel
operators and their applications}.
\newblock Springer Monographs in Mathematics. Springer-Verlag, New York,
2003.

\bibitem{Putinar}
M.~Putinar.
\newblock Positive polynomials on compact semi-algebraic sets.
\newblock \href{http://dx.doi.org/10.1512/iumj.1993.42.42045}%
{\em Indiana Univ. Math. J.}, 42(3):969--984, 1993.

\bibitem{roberts-varberg}
A.W.~Roberts and D.E.~Varberg.
\newblock {\em Convex functions}.
\newblock Academic Press [A subsidiary of Harcourt Brace Jovanovich,
Publishers], New York-London, 1973.
\newblock Pure and Applied Mathematics, Vol. 57.

\bibitem{Rudin59}
W.~Rudin.
\newblock Positive definite sequences and absolutely monotonic functions.
\newblock \href{http://dx.doi.org/10.1215/S0012-7094-59-02659-6}%
{\em Duke Math. J}, 26(4):617--622, 1959.

\bibitem{SSV}
R.L.~Schilling, R.~Song, and Z.~Vondracek.
\newblock {\em Bernstein Functions. Theory and Applications}.
\newblock De Gruyter Studies in Mathematics, 37. De Gruyter, Berlin,
2012.

\bibitem{Schmuedgen} 
K.~Schm\"udgen.
\newblock \href{http://dx.doi.org/10.1007/978-3-319-64546-9}{\em The
moment problem.}
\newblock Graduate Texts in Mathematics, 277. Springer, 2017.

\bibitem{Schoenberg33}
I.J.~Schoenberg.
\newblock On finite-rowed systems of linear inequalities in infinitely
many variables. II.
\newblock \href{http://dx.doi.org/10.2307/1989776}%
{\em Trans. Amer. Math. Soc.}, 35(2):452--478, 1933.

\bibitem{Schoenberg42}
I.J.~Schoenberg.
\newblock Positive definite functions on spheres.
\newblock \href{http://dx.doi.org/10.1215/S0012-7094-42-00908-6}%
{\em Duke Math. J.}, 9(1):96--108, 1942.

\bibitem{Schoenberg51}
I.J.~Schoenberg.
\newblock On P\'{o}lya frequency functions. I.~The totally positive
functions and their Laplace transforms.
\newblock \href{https://doi.org/10.1007/BF02790092}%
{\em J.\ Analyse Math.}, 1:331--374, 1951.

\bibitem{Schur1911}
J.~Schur.
\newblock {B}emerkungen zur {T}heorie der beschr{\"a}nkten
{B}ilinearformen mit unendlich vielen {V}er{\"a}nderlichen.
\newblock \href{http://dx.doi.org/10.1515/crll.1911.140.1}%
{\em J.~reine angew.~Math.}, 140:1--28, 1911.

\bibitem{STmoment}
J.A.~Shohat and J.D.~Tamarkin.
\newblock {\em The problem of moments}.
\newblock AMS Mathematical Surveys, American Mathematical Society,
New York, 1943.

\bibitem{Simon}
B.~Simon.
\newblock \href{http://dx.doi.org/10.1007/978-3-030-22422-6}{\em
Loewner's Theorem on Monotone Matrix Functions}.
\newblock Springer Nature Switzerland, 2019.

\bibitem{Stochel85}
J.~Stochel.
\newblock Decomposition and integral representation of covariance
kernels.
\newblock {\em Bull.\ Acad.\ Polon.\ S\'er.\ Sci.\ Math.},
33:367--376, 1985.

\bibitem{Stochel91}
J.~Stochel.
\newblock Smooth positive definite functions on some multiplicative
semigroups.
\newblock \href{http://dx.doi.org/10.1007/BF02844685}%
{\em Rend.\ Circ.\ Mat.\ Palermo}, 40:153--176, 1991.

\bibitem{Stochel92}
J.~Stochel.
\newblock Decomposition and disintegration of positive definite
kernels on convex $*$-semigroups.
\newblock {\em Ann.\ Polon.\ Math.}, 56(3):243--294, 1992.

\bibitem{Stochel}
J.~Stochel and J.B.~Stochel.
\newblock On the $\varkappa$th root of a Stieltjes moment sequence.
\newblock \href{http://dx.doi.org/10.1016/j.jmaa.2012.07.012}%
{\em J. Math. Anal. Appl.}, 396(2):786--800, 2012.
 
\bibitem{Titchmarsh} 
E.C.~Titchmarsh.
\newblock {\em The theory of functions}.
\newblock Oxford University Press, Oxford, Reprint of the second
(1939) edition, 1958.

\bibitem{vasudeva79}
H.L.~Vasudeva.
\newblock Positive definite matrices and absolutely monotonic functions.
\newblock {\em Indian J. Pure Appl. Math.}, 10(7):854--858, 1979.

\bibitem{Widder34}
D.V.~Widder.
\newblock Necessary and sufficient conditions for the representation
of a function by a doubly infinite Laplace integral.
\newblock \href{https://projecteuclid.org/euclid.bams/1183497370}%
{\em Bull. Amer. Math. Soc.}, 40(4):321--326, 1934.
  
\bibitem{Widder}
D.V.~Widder.
\newblock {\em The Laplace transform}.
\newblock Princeton University Press, Princeton, first edition, 1941.

\end{thebibliography}
\end{document}